\documentclass[journal,twoside,web]{ieeecolor}
\usepackage{generic}
\usepackage{cite}
\usepackage {tablefootnote}
\usepackage{amsmath,amssymb,amsfonts}
\usepackage{extarrows}
\usepackage{algorithmic,algorithm}
\usepackage[T1]{fontenc}
\usepackage{graphicx}
\usepackage{makecell}
\usepackage{adjustbox}
\usepackage{hyperref}
\usepackage{cleveref}
\usepackage{ntheorem, pifont}
\usepackage{color}
\usepackage{textcomp}

\allowdisplaybreaks
\def\BibTeX{{\rm B\kern-.05em{\sc i\kern-.025em b}\kern-.08em
		T\kern-.1667em\lower.7ex\hbox{E}\kern-.125emX}}
\begin{document}
	\title{Distributed Fractional Bayesian  Learning for Adaptive Optimization}
	\author{Yaqun Yang, Jinlong  Lei ({\it Member, IEEE}), Guanghui Wen ({\it Senior Member, IEEE}),  and Yiguang Hong ({\it Fellow, IEEE})
		\thanks{The paper is sponsored by the National Key Research and Development Program of China under No 2022YFA1004701,  the National Natural Science Foundation of China under  No. 62373283, and the Fundamental
			Research Funds for the Central Universities, and partially by Shanghai Municipal Science and Technology Major Project No. 2021SHZDZX0100, and National Natural Science Foundation of China (Grant No. 62088101). }			
			\thanks{Yaqun Yang is with the Department of Control Science and Engineering, Tongji University, Shanghai, 201804, China. (email: yangyaqun@tongji.edu.cn) }
  	\thanks{Jinlong Lei  and  Yiguang Hong are with the Department of Control Science and Engineering, Tongji University, Shanghai, 201804, China;  also with State Key Laboratory of Autonomous Intelligent Unmanned Systems,  and Frontiers Science Center for Intelligent Autonomous Systems, Ministry of Education, and the Shanghai Institute of Intelligent Science and Technology, Tongji University, Shanghai, 200092, China. (email: leijinlong@tongji.edu.cn, yghong@iss.ac.cn)}
    	\thanks{Guanghui Wen is with the Department of Systems Science, School of Mathematics, Southeast University, Nanjing, 211189, China. (email: ghwen@seu.edu.cn)}}
	
	\maketitle
	\begin{abstract}
		  This paper considers a  distributed  adaptive optimization problem, where all agents  only have  access to their local cost functions  with a common unknown parameter, whereas they mean to collaboratively estimate the true parameter  and find the optimal solution over a connected network.   A general mathematical framework for such a problem has not been studied yet. We aim to provide valuable insights for addressing parameter uncertainty in distributed optimization problems and simultaneously find the optimal solution. Thus, we propose a novel distributed  scheme,  which utilizes distributed fractional Bayesian learning
		  through weighted  averaging on  the log-beliefs
		  to update  the  beliefs of  unknown parameter, and  distributed gradient descent for renewing the estimation of  the optimal solution.  Then under suitable assumptions,  we prove that all agents'  beliefs and    decision variables  converge almost surely to  the true parameter  and the optimal solution under the true parameter, respectively. We further establish   a sublinear  convergence  rate  for the belief sequence. Finally, numerical experiments are implemented to corroborate the  theoretical analysis.
	\end{abstract}
	
	\section{Introduction}
	\subsection{Backgrounds and Motivations}
	Distributed optimization has been widely used for modeling and resolving cooperative decision-making problems in large-scale multi-agent systems including economic dispatch, smart grids, automatic controls, and machine learning (see e.g., \cite{ML1, ML2}).  However, in many complex situations, agents need to make decisions with uncertainty. For example, in Robotics, planning the task for a robot requires predicting other agents' reactive behaviors which might be unknown at the very beginning \cite{trautman2010unfreezing}, while autonomous vehicles need to interpret the intentions of others and make trajectory planning for itself\cite{ chahine2023intention}. 
	Also, for the Markowitz profile problem, one should learn the uncertain parameters of expectation or covariance matrices associated with the stocks model, and then find the best solution to the optimal portfolio \cite{cornuejols2006optimization}.  These challenges together motivate us to investigate distributed decision-making problems involving model uncertainty.

	Generally speaking, the resolution of decision-making problems with model uncertainty consists of two processes: model construction and decision making\cite{wilson2018adaptive}, i.e., agents need to estimate the unknown model function (the classical setting is characterized by known function structure while with unknown parameters) and find the optimal solution to it.  The commonly used approaches include the sequential and simultaneous methods. However, a sequential method that considers {\it optimization after prediction}  may not be applicable to complex decision-making scenarios, since large-scale parameter learning problems lead to a long waiting time for solving the original problem. Besides, as has been analyzed in \cite{oco},  this scheme provides an approximate solution to model parameters, which propagates the corrupt error into the objective optimization. In some practical scenarios,  {\it optimization after prediction}
	may lead to a ``frozen robot" problem as pointed out in \cite{le2021lucidgames}. Therefore, developing dynamic learning coupled algorithm that consider  {\it prediction while optimization} is crucial and has gained increasing popularity in recent years, see e.g., \cite{best2015bayesian,le2021lucidgames,oco}.

	It is noticed that the aforementioned works  \cite{trautman2010unfreezing, chahine2023intention, cornuejols2006optimization,best2015bayesian,le2021lucidgames} investigate the coupled phenomenon between model construction and decision making in specific scenarios and develop corresponding methods. Whereas the previously studied theoretical works
	mostly focus on the centralized problem with parameter uncertainty, and merely consider  the unidirectional coupling  of optimization and prediction where the estimation of the model parameter is independent of decision making\cite{jiang2016imperfect,mislagrange,miscen,distributedmis}. Moreover,  few of them have investigated the large-scale distributed scenarios.
In our work, we focus on the \textit{prediction while optimization} loop and rigorously derive the convergence of both model parameter estimation and optimal decision in large-scale distributed scenario. 
	

We  consider  the bidirectional coupling of parameter learning and objective optimization, which brings more difficulties to the resolutions along with theoretical analysis. Though there exist some related works, most of them assume that the unknown parameter influences the objective function in a specific structure. For example,  \cite{personalizedgra_traking} considers a distributed quadratic optimization problem with the unknown model parameter being the objective coefficients, and  adopts  the recursive least square  to estimate the parameter and gradient tracking to solve the objective optimization.
While the work \cite{personalizedgra_traking} imposes some assumptions on the intermediate process, which however is lack of strict theoretical verification. As a result, bidirectional coupling optimization problem has   not  been fully resolved here.
 In addition, \cite{guolei} uses weighted least square  to solve the unknown coefficient matrices in linear-quadratic stochastic differential games.  Our work is different from those problems which impose   particular structures on objective functions.
 
%
	
	%
		\vspace{-5pt}
		\subsection{Problem Formulation and Challenges}
		We characterize the uncertainty of the distributed optimization problem  in a parametric sense. Our primary objective is to establish the model parameters in a way that the action generated from this estimated model best matches the observed action, meanwhile, find this best-estimated model's optimal solution.
			To be specific,  we consider a distributed optimization problem with unknown model parameter $\theta$ as follows.
		\begin{align}
			&\min_{x\in \mathbb{R}} ~\frac{1}{N}\sum_{i=1}^{N} J_i(x,\theta_*), \theta_*\in \Theta, \label{problem}
		\end{align}
		where $J_i(x,\theta_*)$ represents the private cost function of  agent $i\in\mathcal{N}:=\{1,2,\cdots,N\}$. The unknown true parameter $\theta_*$  is taken  from a finite set $\Theta:=\{\theta_1,\theta_2,\cdots,\theta_M\}$. Note that we do not restrict the form of the parameters $\theta$ within the abstract function $J(x, \theta)$.  This flexibility allows our model to be applicable to a broader range of problems, enhancing its generality and adaptability. 
	
		Each agent $i\in \mathcal{N}$  has a prior belief $q_i(\theta_m), m=1,2,\cdots, M$ of the $M$ possible parameters. Given an input strategy $x$, the feedback is realized randomly from a probability distribution depending on  the system's true parameter, i.e. the noisy feedback $y_i=J_i(x,\theta_*)+\epsilon_i$ for every agent $i$.  
			 Let $f_i(y_i|x,\theta_m)$ denote the likelihood function (also called probability density function here) of observation $y_i$ for any strategy $x\in \mathbb{R}$ under parameter $\theta_m\in \Theta$. For example, if  $\epsilon_i\sim N(0, \sigma^2)$, the likelihood function
			 of observation $y$ under input $x$ and parameter $\theta_m$ should be
			 \begin{align*}
			 	f_i(y_i|x,\theta_m)=\frac{1}{\sigma\sqrt{2\pi}}\exp\left[-\frac{(y_i-J_i(x,\theta_m))^2}{2\sigma^2}\right].
			 \end{align*} 
		This setting is aligned with \cite[Example 3]{wu2020multi} and \cite[Section II]{nedic2017fast}.
		


%
			
			Though each agent only knows its local information, it can interact with other agents over a fixed connected network $\mathcal{G}=\{\mathcal{N},\mathcal{E}, W\}$ in which $\mathcal{N}=\{1,2,\cdots,N\}$ is the set of agents.   Herein, $\mathcal{E}\subseteq N\times N$ represents the edges of network, where $(i,j) \in \mathcal{E}$ if and only if agents $i$ and $j$ are connected. Each agent $i$ has a set of neighbors $\mathcal{N}_i=\{j|(i,j)\in\mathcal{E}\}$. $W=[w_{ij}]_{N\times N}$ denotes the weighted adjacency matrix, where $w_{ij}>0$  if $j\in \mathcal{N}_i$ and $w_{ij}= 0$ otherwise.
 The agents want to collaboratively  solve the problem \eqref{problem},  namely, simultaneously
find the true parameter $\theta_*\in \Theta$ and the optimal solution $x_*$ to the global objective function.
		
There are several challenges to solving this problem. Firstly, we need to develop a fully distributed strategy
based on local information and local communication. This approach is significantly more challenging than dealing with centralized issues like \cite{jiang2016imperfect,mislagrange,miscen}. Secondly,  given the parameter uncertainty in the objective optimization problem,
since the sequential method cannot attain an exact solution, we need to design a scheme that simultaneously estimates the parameter and find the optimal solution.  Thirdly, the process of simultaneously learning and optimizing the objective function is coupled in both directions, and the existence of stochastic noises will bring about difficulties in the rigorous theoretic analysis of the designed scheme.
	All in all, these challenges highlight the complex and dynamic nature of addressing such problems.
		This paper  addresses   all the aforementioned challenges associated with  the problem \eqref{problem}, and will summarize the main contributions
in section \ref{contribution}.
		%
		
		\subsection{Related Works}
		
		The theory of distributed optimization has been extensively studied, including 
		convex or non-convex objective functions, smooth or non-smooth conditions, static or time-varying networks (see e.g.,\cite{nedic2009distributed,zhu2011distributed}).  By 2020,  survey  papers related to this field have appeared one after another like \cite{survey2, yang2019survey}.
Most of the  distributed optimization works considered   precisely known objective functions, while seldom of them have  investigated the model uncertainty.

			In  recent years, the problems with both unknown parameter learning and objective optimization have gradually attracted research attention.
			One widely adopted methodology for handling external disturbances and noise is robust optimization, which accounts for uncertainties by optimizing system performance under the worst-case realization of unknown parameters. For instance, \cite{li2021distributionally,li2023distributionally} present centralized robust optimization frameworks, while \cite{fu2025cutting} investigates a distributed formulation.  In comparison,  we consider the problem from a different perspective and propose
			a dual proactive process which contains both the inference of uncertain parameters and the optimization of the objective function. There are some related works.
			For example,  \cite{ahmadi2020resolution} presented a coupled stochastic optimization scheme to solve problems with imperfect information.   \cite{oco} introduced a method to optimize decisions in a dynamic environment, where the model parameter is unavailable but may be learned by a separate process called Joint estimation-optimization.  In addition, \cite{jiang2016imperfect,mislagrange,miscen} considered centralized mis-specific convex optimization problems $f(x,\theta)$, where the unknown parameter $\theta$ of objective is a solution to some learning problem $l(\theta)$.  To be specific,  \cite{jiang2016imperfect} and \cite{miscen} both used the gradient descent method to solve the parameter learning problem and objective optimization problem under deterministic optimization and stochastic optimization scenarios respectively, whereas \cite{mislagrange}  investigated an inexact parametric augmented Lagrangian method to solve such a problem. However, the aforementioned {\it prediction while optimization} works are centralized schemes and unidirectional coupling,
			i.e. the  objective optimization depends on parameter learning while the parameter learning problem is independent of objective optimization. Although distributed coupled optimization has also been investigated, for example,  \cite{distributedmis} proposed a  distributed stochastic optimization with imperfect information and  \cite{personalizedgra_traking} presented a distributed problem with a composite structure consisting of an exact engineering part and an unknown personalized part.  However,  \cite{distributedmis}  still  focused on unidirectional coupling, while   \cite{personalizedgra_traking} imposed some assumptions on the the intermediate process.

%
		
		
		It is worth noting that  the coupling between parameter learning  and equilibrium searching   have been little investigated in the field of game theory. For example, \cite{huang2023distributed} considered parameter learning and decision-making in game theory and developed a non-Bayesian method for parameter estimating.   Moreover, \cite{wu2020multi} examined the learning dynamics influenced by strategic agents engaging in multiple rounds of a game with an unknown parameter that affects the payoff, although this paper operates under the centralized scheme.
		
		Inspired by \cite{wu2020multi}, we consider using a Bayesian type scheme to learn the model parameter.  Bayesian inference is widely used in belief updating of uncertainty parameters\cite{bissiri2016general,wu2020multi}.   The standard Bayesian method fully generates past observations to update the parameter estimation. In comparison, 
		fractional Bayesian methods have emerged as a powerful tool in statistical inference and machine learning, offering robustness to model misspecification and enhanced flexibility in handling complex data structures \cite{krog2018bayesian,erazo2024parameter}. These methods modify the traditional Bayesian framework by incorporating fractional likelihood functions, which can significantly improve the performance of inference under model uncertainty \cite{molavi2017foundations,epstein2010non,bhattacharya2019bayesian}.
		In addition, 
		The authors of \cite{lalitha2018social} have shown that in distributed learning,  the fractional Bayesian inference with distributed log-belief consensus can get a fast convergence rate.  As such, we consider this variation of Bayesian inference to estimate the unknown parameter of our problem. As for the adaptive optimization method with the objective function computing, we consider the classical distributed gradient descent \cite{pu2021sharp}, which also has good performance in convex optimization.  
	
Furthermore, in order to clearly explain how our proposed framework differs from the most related existing approaches, we provide Table \ref{tab}.
	\begin{table*}[h]
		\footnotesize
		\caption{Work comparison with previous studies of \textit{Prediction while Optimization} settings}	\label{tab}
		\begin{adjustbox}{center}
			\begin{tabular}{|c|c|c|}
				\hline
				Problem & Algorithm & Explanations \\
				\hline 
				\makecell{$\min_{x} f(x,\theta_*), \theta_*\in \arg\min_{\theta} l(\theta)$} & \makecell{\textbf{Parameter prediction:}\\$
					\theta_{k+1}=\theta_{k}-\alpha \nabla l(\theta_{k})$\\
					\textbf{Decision optimization:}\\$x_{k+1} = x_k -\beta \nabla_x f(x_k,\theta_{k})$}  & \makecell{Unidirectional coupling algorithm; \\some works on centralized scenario \cite{jiang2016imperfect,mislagrange,miscen},\\ few works on distributed scenario \cite{distributedmis,mywork1}, \\while  distributed scenario uses linear consensus. } \\ 
				\hline
				\makecell{$\min_x \frac{1}{2}x^T P x+q^Tx +r$\\ with unknown $P, q,r$;\\
					Realized data  $(x_s, y_s)_{s=1}^t$, where \\$ y_s=\frac{1}{2}x_s^T P x_s+q^Tx_s +r+ \epsilon$} &  \makecell{\textbf{Parameter prediction:}\\ Use data $(x_s, y_s)_{s=1}^t$ to do recursive least square\\
					\textbf{Decision optimization:} \\Gradient descent based on current \\ parameter $\{\hat{P}_t, \hat{q}_t, \hat{r}_t\}$ } & \makecell{Only for quadratic problem,\\ so that use least square for parameter\\ learning  \cite{personalizedgra_traking, guolei}} \\
				\hline 
				\makecell{\textbf{Our work: } \\
					$\min_x \frac{1}{N}\sum_{i=1}^{N} J_i(x,\theta_*)$\\ $\theta_*\in \{\theta_1,\theta_2,...,\theta_M\}$}
				& \makecell{\textbf{Parameter prediction:} \\ Use feedback data do fractional Bayesian learning \\
					\textbf{Decision optimization:} \\Gradient descent based on average function\\ of all possible parameter}  &  \makecell{Applicable to a broader range of problems \\ instead of special objective structure; \\Bidirectional  coupling algorithm bring difficulties; \\ Consensus averaging on a \\ reweighting of the log-belief.}\\
				\hline
			\end{tabular}
		\end{adjustbox}
	\end{table*}
		
		\subsection{Main Contribution}
		\label{contribution}
To solve the distributed optimization problem  \eqref{problem} with  unknown parameter $\theta_*$, we design an efficient algorithm and give its convergence analysis. Below are our contributions.
		\begin{enumerate}
			\item We propose a general mathematical formulation for distributed  optimization problem with parameter uncertainty.
The formulation  models the bidirectional coupling between parameter learning and objective optimization.
 Though there   has been  a few research on some  practical applications, the general mathematical model has not been abstracted and studied yet. Thus, our formulated model can expand  upon prior theoretical  works with known objective functions, and the type with  fixed model structure influenced by unknown parameter which however   is independent of the  objective computation.
			\item We design a novel distributed fractional Bayesian learning dynamics and adaptive optimization algorithm, which considers model construction and decision-making simultaneously in the {\it Prediction while Optimization} scheme.  To be specific, we use fractional Bayesian learning for updating beliefs of  the unknown parameter, which adopts a distributed consensus protocol that averages on a reweighting of the log-belief for the belief consensus.  This is more reasonable  and robust than standard Bayesian learning, and the  belief consensus protocol   is shown to be faster than the normal distributed linear consensus protocol by experiment.
We then utilize  the distributed gradient descent method  to update the optimal solution, whereas each agent's  gradient is computed based on  the expectation of  its local  objective  function over its private belief.
			
			\item  Finally, we rigorously prove that  all agents' belief converge almost surely  to a common belief  that is consistent with the   true parameter,	and that  the decision variable of every agent converges to the optimal result under this common true belief. Besides, we also give the convergence rate analysis of belief.
		\end{enumerate}

		\section{Algorithm and Assumptions}
		
		In this section, we propose a distributed fractional Bayesian learning method to solve the problem (\ref{problem}) with some basic assumptions.
		
		\subsection{Algorithm Design}
		To solve the problem (\ref{problem}),  we  need to update the belief of the unknown parameter set $\Theta$, and  get the adaptive decision based on the current belief. 
		At each step $t$, every agent $i\in \mathcal{N}$ maintains its private belief $q_i^{(t)}$ and  local decision $x_i^{(t)}$.
		Firstly,  each agent updates its belief by Bayesian fractional posterior \eqref{b_up} based on its current observation,   exchanges information with its  neighbors $\mathcal{N}_i=\{j|(i,j)\in\mathcal{E} \}$ over the distributed network $\mathcal{G}$ and performs a non-Bayesian consensus  using log-beliefs  \eqref{q_up} to renew  the  belief $q_i^{(t+1)}$.
		Secondly, we  obtain an adaptive decision based on the updated belief.
		Each agent $i\in \mathcal{N}$ calculates a local function $\sum_{\theta \in \Theta} J_i\left(x_i^{(t)}, \theta\right) q_i^{(t+1)}(\theta)$ by  averaging its private cost function $J_i (x , \theta  )$ across its    belief  $q_i^{(t+1)}$,  and then performs a  gradient descent method based on this local function  and  shares the intermediate result with its neighbors. This formate a communication after computation form, which is quiet common in all-reduce distributed algorithm\cite{bluefog}.  After  receiving  its neighbors' temporary decision information over the static connected network,  agent $i\in \mathcal{N}$ renews the decision $x_i^{(t+1)}$ by a distributed linear consensus protocol.
		Finally, we feed the results of the current iteration into the unknown system to obtain the corresponding output data with noise and proceed to the next loop. The pseudo-code for the algorithm is outlined in Algorithm \ref{alg:CDSA}.
		
		\begin{algorithm}
			\caption{Distributed Fractional Bayesian Learning in Optimization}
			\label{alg:CDSA}
			\begin{algorithmic}
				\STATE{\textbf{Initialization:} For each $i\in\mathcal{N}$: $(x_i^{(0)},y_i^{(0)})$ ; stepsize sequence $\{\alpha^{(t)}\geq0\}_{t\geq0}$; weigh matrix $W=[w_{ij}]_{N\times N}$; prior distribution $q_i^{(0)}=\frac{1}{M}\boldsymbol{1}_M$}
				\STATE{\textbf{Belief update:} for each agent $i\in\mathcal{N}$, and $m=1,\cdots,M$}
				\STATE{\quad Update local fractional Bayesian posterior belief}
				\begin{equation}
					b_i^{(t)}(\theta_m)=\frac{f_i(y_i^{(t)}|x_i^{(t)},\theta_m)^{\alpha(t)}q_i^{(t)}(\theta_m)}{\sum_{\theta\in\Theta}f_i(y_i^{(t)}|x_i^{(t)},\theta)^{\alpha(t)}q_i^{(t)}(\theta)},\label{b_up}
				\end{equation}
				\STATE{\quad Receive information $b_j^{(t)}(\theta_m)$ from $j\in\mathcal{N}_i$ and perform a fractional Bayesian rule to update the private belief}
				\begin{align}		
					q_i^{(t+1)}(\theta_m)=\frac{\exp(\sum_{j\in \mathcal{N}_i}w_{ij}log(b_j^{(t)}(\theta_m)))}{\sum_{\theta\in\Theta}\exp(\sum_{j\in \mathcal{N}_i}w_{ij}log(b_j^{(t)}(\theta)))} \label{q_up}
				\end{align}
				\STATE{\textbf{Decision update:} Given the current private belief $q_i^{(t+1)}$,  each agent $i\in\mathcal{N}$ evaluates its local expected cost by}
				\begin{align}
					\tilde{J}_i(x_i^{(t)},\boldsymbol{\theta})=\sum_{\theta\in\Theta}J_i(x_i^{(t)},\theta)q_i^{(t+1)}(\theta).\label{be_av}
				\end{align}
				\STATE{\qquad Then, perform a local gradient descent and share the current local state with neighboring nodes, namely}
			\begin{align}
				x_i^{(t+1)}= \sum_{j\in \mathcal{N}_i}w_{ij}\left[x_j^{(t)}-\alpha^{(t)}\frac{\partial}{\partial x}	\tilde{J}_j(x_j^{(t)},\boldsymbol{\theta})\right]	\label{decision_update}
			\end{align}
				\STATE{\textbf{Obtain the new data:}  Every agent $i\in\mathcal{N}$   gets   new data based on the renewed  decision under true parameter $\theta_*$.}
				\begin{align}
					y_i^{(t+1)}=J_i(x_i^{(t+1)},\theta_*)+\epsilon_i.\label{y_data}
				\end{align}
			\end{algorithmic}
		\end{algorithm}
		
		\textbf{Remark 1. } Compared to the standard Bayesian   posterior  in multi-agent Bayesian learning  \cite{wu2020multi}, we use Bayesian fractional posterior distribution in (\ref{b_up}). It has been demonstrated to be valuable in Bayesian inference because of its flexibility in incorporating historical information. This method modifies the likelihood of historical data using a fractional power $\alpha^{(t)}$\cite{han2023study}. The parameter $\alpha$ controls the relative weight of loss-to-data to loss-to-prior. If $0<\alpha< 1$,  the loss-to-prior is given more prominence than newly generated data in the Bayesian update; $\alpha=1$ is the standard Bayesian; $\alpha>1$ that means we pay more attention to data,  and in the extreme case  with large $\alpha$, the Bayesian estimator degenerates into maximum likelihood estimator as in frequentist inference \cite{bissiri2016general}. It has been shown in \cite{grunwald2012safe} that for small $\alpha,$  fractional Bayesian inference outperform standard Bayesian  for the underlying unknown distribution in several settings.
		
		
		\textbf{Remark 2.}   Different from the standard linear consensus in distributed scenarios \cite{nedic2016tutorial},
		we  adopt  (\ref{q_up})  that implements distributed consensus averaging on a reweighting of the log-beliefs.
		It is worth noting that  the standard linear consensus protocol
		simplified into a vector form $\boldsymbol{x}(t+1)=\boldsymbol{W}\boldsymbol{x}(t)$ \cite{fastdis_ave} has a convergence rate of $\mathcal{O}(\rho_w^t)$, where $\rho_w$ is the spectral radius of $W-\frac{\boldsymbol{1}\boldsymbol{1}^T}{N}$. Log-belief consensus $\log \boldsymbol{x}(t+1)=\boldsymbol{W}\log \boldsymbol{x}(t)$ can be recast  as  $ \boldsymbol{y}(t+1)=W\boldsymbol{y}(t)$ with $\boldsymbol{y}(t)\triangleq\log\boldsymbol{x}(t)$, where $\boldsymbol{y}(t)$ converges at rate $\mathcal{O}(\rho_w^t)$, hence $\boldsymbol{x}(t)$ displays a exponential faster  rate than $\boldsymbol{y}(t)$. Thus, the utilized  method (\ref{q_up})  is likely to bring    a  faster rate of consensus.
	
\textbf{Remark 3.}  This   work primarily  focuses on unknown parameter in a discrete set $\Theta=\{\theta_1,\theta_2,...,\theta_M\}$, while it might  be potentially  extended into continuous parameter case.
As for continuous bounded set $\Theta$, the general update of fractional Bayesian posterior belief  in (\ref{b_up}) should be $g^{(t+1)}(\theta)=\frac{f(y^{(t)}|x^{(t)},\theta)^{\alpha(t)}g^{(t)}(\theta)}{\int_{\theta\in\Theta}f(y^{(t)}|x^{(t)},\theta)^{\alpha(t)}g^{(t)}(\theta)}$ for all $ \theta\in\Theta$, where $g^{(t)}(\theta)$ is the probability density function of $\theta$ on the set $\Theta$ at time $t$. Since computing the full posterior belief in each step, which involves continuous integration in the denominator, can be computationally intensive.  It is possible to use the Maximum A Posteriori (MAP) estimator $g_M^{(t)}=\arg\max_{\theta\in\Theta} g^{(t)}(\theta)= \arg \max_{\theta\in\Theta}  g^{(1)}(\theta) \prod_{j=1}^{k-1} f(y^{(t)}|x^{(t)},\theta)^{\alpha(t)}$ as a shift. See \cite[Section 6]{wu2020multi} for more details. 
		
		\subsection{Assumptions}
		To prove  the convergence of sequences $\{x_i^{(t)}\}_{t\geq 0}$ and $\{q_i^{(t)}\}_{t\geq 0}$ generated by Algorithm \ref{alg:CDSA} for all agents $i\in \mathcal{N}$, we give some assumptions as follows.
		\newtheorem{assum}{\textbf{Assumption}}
		\begin{assum}[Bounded Belief]\label{assum_boundbelief}
			Every realized cost has bounded information content, i.e.,  there exists a positive constant $B$ such that
			\begin{equation}
				\max_{i}\max_{\theta^{'},\theta^{''}\in\Theta}\max_{x}\sup_{y_i}\left|\log{\frac{f_i(y_i|x,\theta^{'})}{f_i(y_i|x,\theta^{''})}}\right|< B
			\end{equation}
			In addition, for each $i\in\mathcal{N}$, $f_i(y_i|x,\theta)$ is continuous in $x$ for all $\theta\in\Theta$.
		\end{assum}

	  Bounded private beliefs suggest that an agent $i\in\mathcal{N}$ can only reveal a limited amount of information about the unknown parameter. Conversely, the unbounded belief   $\sup_{y_i}\left|\log{\frac{f_i(y_i|x,\theta^{'})}{f_i(y_i|x,\theta^{''})}}\right|=\infty$ corresponds to a situation where an agent may receive arbitrarily strong signals favoring the true parameter\cite{acemoglu2011bayesian}.
		In this case, the information of agent $i$ is enough for  revealing the true parameter, and hence it is unnecessary to use the observation of multiple agents. Therefore, Assumption \ref{assum_boundbelief} is imposed to preclude the degraded case and make the multi-agent setting meaningful. 
		
		
		\begin{assum}[Graph and Weighted Matrix]\label{assum_graph}
			The  graph $\mathcal{G}$ is static, undirected and connected. The weighted adjacency matrix W is nonnegative and doubly stochastic, i.e.,
			\begin{equation}
				W\boldsymbol{1}=\boldsymbol{1}, \boldsymbol{1}^TW=\boldsymbol{1}^T.
			\end{equation}
		\end{assum}
	
This assumption is crucial in the development of distributed algorithms, based on which every agent's information can be merged after multiple rounds of communication. Then  consensus will be obtained.  With  Assumption \ref{assum_graph}, we can get the following lemma  from \cite{fastdis_ave}.
		\newtheorem{lemma}{\textbf{Lemma}}
		\begin{lemma}\cite[Theorem 1]{fastdis_ave}\label{W_lim}
			Let Assumption \ref{assum_graph} hold. Then $$\lim_{t\rightarrow\infty}W^t=\frac{\boldsymbol{1}\boldsymbol{1}^T}{N}$$ holds with exponential rate  $\mathcal{O}(\rho_w^t)$, where $ \rho_w \in [0,1)$ is the  the spectral radius of  $W-\frac{\boldsymbol{1}\boldsymbol{1}^T}{N}$.
		\end{lemma}
		\begin{assum}[Stepsize Policy]\label{assum_step}
			The stepsize sequence $\{\alpha^{(t)}\}_{t\geq0}$  with $0<\alpha^{(t)}<1$ satisfies
 \[\sum_{t=0}^{\infty}\alpha^{(t)}=\infty {\rm~ and ~}\sum_{t=0}^{\infty}(\alpha^{(t)})^2\leq \infty.\]
		\end{assum}
		
		This assumption indicates that $\lim_{t \rightarrow \infty}\alpha^{(t)}=0$.
		
		 In the following, we impose some assumptions regarding  the strong  convexity  and Lipschitz smooth on the cost functions. 
			
			\begin{assum}[Function Properties]\label{assum_func}
				For every $i\in\mathcal{N}$, $J_i(x,\theta)$  is strongly convex and Lipschitz smooth in $x$ with constant $\mu$ and $L$ for any fixed $\theta  \in  \Theta$, i.e., for any $x,x'\in \mathbb{R}$, we have
				\begin{align}
					\left(\nabla_x J_i\left(x^{\prime}, \theta\right)-\nabla_x J_i(x, \theta)\right)^T\left(x^{\prime}-x\right) & \geq \mu\left\|x^{\prime}-x\right\|^2, \notag \\
					\left\|\nabla_x J_i\left(x^{\prime}, \theta\right)-\nabla_x J_i(x, \theta)\right\| & \leq L\left\|x^{\prime}-x\right\|.\notag
				\end{align}
			\end{assum}
			
%
			
			Finally, we impose the following  condition  on  the likelihood function $f_i(y_i|x,\theta)$ (viz. Probability Density Function),  which   can  guarantee the uniqueness of  true parameter $\theta_*$.
			\begin{assum}[Uniqueness of  true parameter $\theta_*$ ]\label{assum_optimal}
				For every $\theta\neq \theta_*$, there exists at least one agent $i\in \mathcal{N}$ with the KL divergence $D_{KL}\left(f_i(y_i|x,\theta_*)||f_i(y_i|x,\theta)\right)>0$ for all $x\in \mathbb{R}$. Here,  the  KL divergence between the distribution of observed $y$ with decision $x$ under parameter $\theta_*$ and $\theta\in\Theta$ is given by
				\begin{align}
					D_{KL}\!\left(f(y|x,\theta_*)||f(y|x,\theta)\right)\!=\!\int_{y}\!f(y|x,\theta_*)log\!\left(\!\frac{f(y|x,\theta_*)}{f(y|x,\theta)}\!\right)\!dy.\notag
				\end{align}
			\end{assum}
			
			\section{Convergence Analysis}
			In this section, we give the convergence analysis of  Algorithm \ref{alg:CDSA}.  We not only  show the convergence of the belief $q_i^{(t)}(\cdot)$ about  the unknown parameters, but also  present the convergence analysis of the decision variable $x_i^{(t)}$. The overall proof process is illustrated in figure \ref{fig_proof} for ease of reading. 
			\begin{figure*}[h]
			\centerline{\includegraphics[width=1.7\columnwidth]{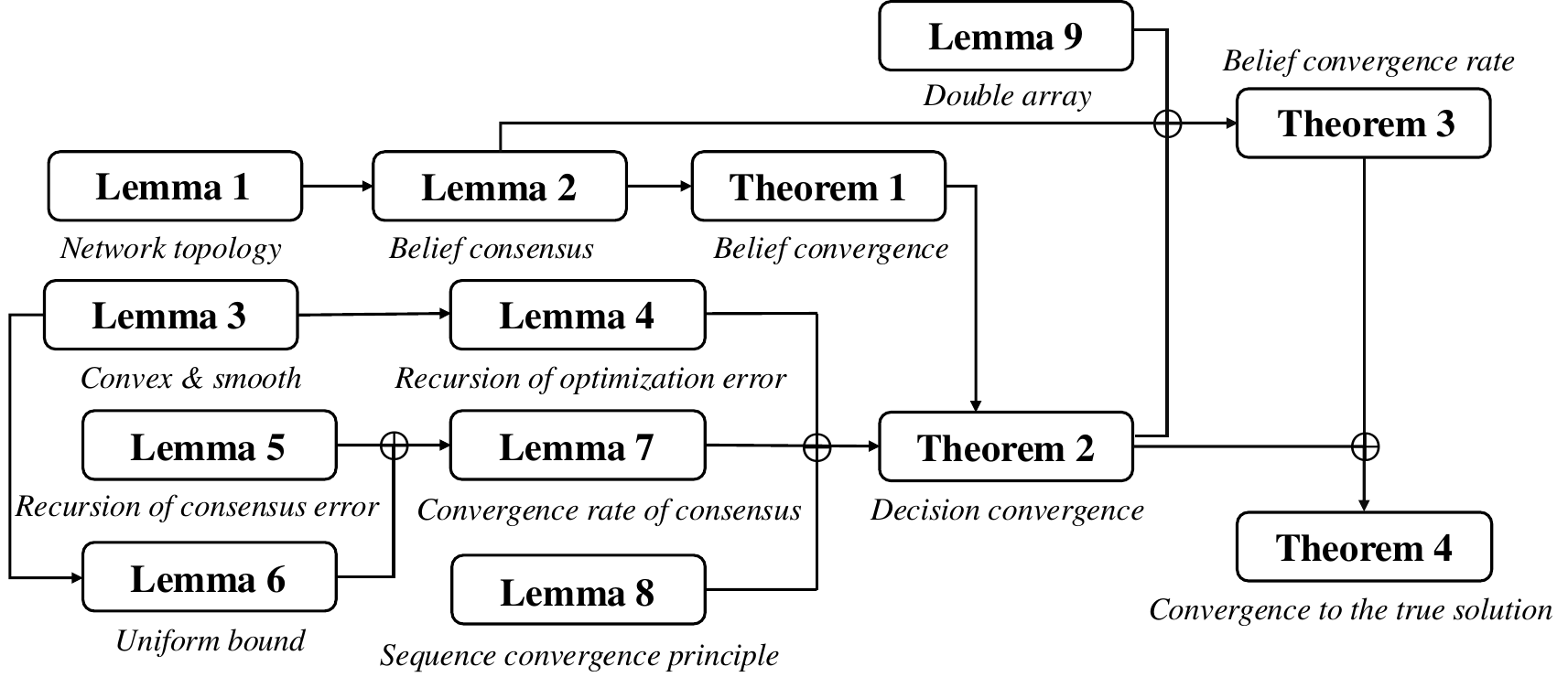}}
			\caption{The overall proof process of this paper, including belief convergence analysis and decision convergence analysis.}
			\label{fig_proof}
		\end{figure*}
			\subsection{Belief Convergence}\label{sec_beliefcon}
			In this subsection, we demonstrate that all agents' beliefs of $\Theta$ converge to a shared belief and present its formula. Though the proof  is motivated by  \cite{huang2023distributed}, observations
are different in  optimization versus  game settings.
So, we include  it  here for completeness.

			
			\begin{lemma}\label{lemma_logcon}
				Let Assumptions \ref{assum_boundbelief}, \ref{assum_graph} and \ref{assum_step} hold. Then   the  agents' log-belief ratios will finally reach consensus, i.e.
				\begin{align}
					\left|\log \tfrac{q_i^{(t+1)}\left(\theta_m\right)}{q_i^{(t+1)}\left(\theta_*\right)}-\tfrac{1}{N} \sum_{i=1}^N \log \tfrac{q_i^{(t+1)}\left(\theta_m\right)}{q_i^{(t+1)}\left(\theta_*\right)}\right|\rightarrow0, \forall \theta_m \in\Theta. \label{con_log}
				\end{align}
			Furthermore, for all $\theta\in\Theta$,  the sequence $\frac{1}{N}\sum_{i=1}^{N}\tfrac{q_i^{(t)}\left(\theta_m\right)}{q_i^{(t)}\left(\theta_*\right)}$  converges  almost surely to some non-negative random variable $\nu_m$.
			\end{lemma}
			\begin{proof}
				According to the belief update rules (\ref{b_up}) and (\ref{q_up}), we have
				\begin{align}\label{log_qq}
					&\log \frac{q_i^{(t+1)}\left(\theta_m\right)}{q_i^{(t+1)}\left(\theta_*\right)} \overset{(\ref{q_up})}{=}\log   \frac{\exp( \sum_{j=1}^N w_{i j}  \log b_j^{(t)} (\theta_m) )}{\exp ( \sum_{j=1}^N w_{i j}  \log b_j^{(t)} (\theta_* ))}   \notag
					\\&  =\sum_{j=1}^N w_{i j}  \log b_j^{(t)} (\theta_m) -  \sum_{j=1}^N w_{i j} \log  b_j^{(t)} (\theta_* ) \notag\\
					&=
					\sum_{j=1}^N w_{i j} \log \frac{b_j^{(t)}\left(\theta_m\right)}{b_j^{(t)}\left(\theta_*\right)} \notag\\
					& \overset{(\ref{b_up})}{=}\sum_{j=1}^N w_{i j} \log \frac{q_j^{(t)}\left(\theta_m\right)}{q_j^{(t)}\left(\theta_*\right)}+\alpha^{(t)} \sum_{j=1}^N w_{i j} \log \frac{f_j\left(y_j^{(t)} | x_j^{(t)}, \theta_m\right)}{f_j\left(y_j^{(t)} | x_j^{(t)}, \theta_*\right)} \notag\\
					&\xlongequal{\text{recursion}}\sum_{j=1}^N W^{t+1}(i, j) \log \frac{q_j^{(0)}\left(\theta_m\right)}{q_j^{(0)}\left(\theta_*\right)}\notag\\
					&~+\sum_{j=1}^N \sum_{\tau=1}^t W^\tau(i, j) \alpha^{(t-\tau+1)} \log \frac{f_j\left(y_j^{(t-\tau+1)} | x_j^{(t-\tau+1)}, \theta_m\right)}{f_j\left(y_j^{(t-\tau+1)} | x_j^{(t-\tau+1)}, \theta_*\right)}\notag\\
					& =\sum_{j=1}^N \sum_{\tau=1}^t W^\tau(i, j) \alpha^{(t-\tau+1)} \log \frac{f_j\left(y_j^{(t-\tau+1)} | x_j^{(t-\tau+1)}, \theta_m\right)}{f_j\left(y_j^{(t-\tau+1)} | x_j^{(t-\tau+1)}, \theta_*\right)},
				\end{align}
				where $W^t(i,j)$ means the $(i,j)$-element of matrix $W^t$, and the last equality follows from $q_i^{(0)}=\frac{1}{M} \mathbf{1}_M$.
				
				With Assumption \ref{assum_graph},  we achieve the double stochasticity of  $W^t$. Then  based on (\ref{log_qq}),  we have
				\begin{align}
					&\frac{1}{N} \sum_{i=1}^N \log \frac{q_i^{(t+1)}\left(\theta_m\right)}{q_i^{(t+1)}\left(\theta_*\right)}=\notag \\
					&\frac{1}{N} \sum_{i=1}^N \sum_{\tau=1}^t \alpha^{(t-\tau+1)} \log \frac{f_i\left(y_i^{(t-\tau+1)} | x_i^{(t-\tau+1)}, \theta_m\right)}{f_i\left(y_i^{(t-\tau+1)} | x_j^{(t-\tau+1)}, \theta_*\right)}.\label{sum_log_qq}
				\end{align}
				
				Therefore, combining (\ref{log_qq}) and (\ref{sum_log_qq}) yields
				\begin{align}
					& \left|\log \tfrac{q_i^{(t+1)}\left(\theta_m\right)}{q_i^{(t+1)}\left(\theta_*\right)}-\tfrac{1}{N} \sum_{i=1}^N \log \tfrac{q_i^{(t+1)}\left(\theta_m\right)}{q_i^{(t+1)}\left(\theta_*\right)}\right| \notag\\
					\leq & \sum_{j=1}^N\!\sum_{\tau=1}^t \alpha^{(t-\tau+1)}\!\left|W^\tau(i, j)-\tfrac{1}{N}\right|\!\left|\log\!\tfrac{f_j(y_j^{(t-\tau+1)} | x_j^{(t-\tau+1)}, \theta_m)}{f_j(y_j^{(t-\tau+1)} | x_j^{(t-\tau+1)}, \theta_*)}\!\right| \notag\\
					\leq& N B \sum_{\tau=1}^t \alpha^{(t-\tau+1)}\left|W^\tau(i, j)-\tfrac{1}{N}\right|,
				\end{align}
				where the last inequality follows from Assumption \ref{assum_boundbelief}. Denote sequence $\gamma_\tau=|W^\tau(i, j)-\tfrac{1}{N}|$.
				In light of Lemma \ref{W_lim}, we can obtain $\lim_{\tau\rightarrow\infty}\gamma_{\tau}=0$ with exponential rate. This together with Assumption  \ref{assum_step}  brings the asymptotic convergence of $ \left|\log \tfrac{q_i^{(t+1)}\left(\theta_m\right)}{q_i^{(t+1)}\left(\theta_*\right)}-\tfrac{1}{N} \sum_{i=1}^N \log \tfrac{q_i^{(t+1)}\left(\theta_m\right)}{q_i^{(t+1)}\left(\theta_*\right)}\right|$.
				\footnote{\cite[Lemma 7]{nedic2010constrained}  Stepsize sequence $0<\alpha^{(t)}<1$ satisfies $\lim_{t \rightarrow \infty}\alpha^{(t)}=0$ under Assumption \ref{assum_step}.  Besides, $0<\gamma_t<1$ is a scalar sequence satisfies $\lim_{t\rightarrow\infty}\gamma_t=0$ with exponential rate, then $\lim_{t\rightarrow\infty}\sum_{\tau=0}^{t}\alpha^{(t-\tau)}\gamma_\tau$=0.}
				
				As for the convergence of sequence $\frac{1}{N}\sum_{i=1}^{N}\tfrac{q_i^{(t)}\left(\theta_m\right)}{q_i^{(t)}\left(\theta_*\right)}$, recalling the third equality of \eqref{log_qq},  we have
%
				\begin{align}
					\frac{q_i^{(t+1)}\left(\theta_m\right)}{q_i^{(t+1)}\left(\theta_*\right)} & =\exp \left(\sum_{j=1}^N w_{i j} \log \frac{b_j^{(t)}\left(\theta_m\right)}{b_j^{(t)}\left(\theta_*\right)}\right) \notag\\
					& \leq \sum_{j=1}^N w_{i j} \frac{b_j^{(t)}\left(\theta_m\right)}{b_j^{(t)}\left(\theta_*\right)}\notag
					\\& \overset{ (\ref{b_up})}{=}\sum_{j=1}^N w_{i j} \frac{f_j\left(y_j^{(t)} | x_j^{(t)}, \theta_m\right)^{\alpha^{(t)}} q_j^{(t)}\left(\theta_m\right)}{f_j\left(y_j^{(t)} | x_j^{(t)}, \theta_*\right)^{\alpha^{(t)}} q_j^{(t)}\left(\theta_*\right)},\label{q_plus_q}
				\end{align}
				where  the first inequality is followed by $e^{\lambda a+(1-\lambda) b}\leq \lambda e^a+(1-\lambda)e^b$, since $e^x$ is a convex function and $\sum_{j=1}^{N}w_{ij}=1$.
				Furthermore, based on $\sum_{i=1}^{N}w_{ij}=1$, we derive
				\begin{align}
					\tfrac{1}{N}\sum_{i=1}^{N}\frac{q_i^{(t+1)}\left(\theta_m\right)}{q_i^{(t+1)}\left(\theta_*\right)} \leq\tfrac{1}{N}\sum_{i=1}^{N} \tfrac{f_i\left(y_i^{(t)} | x_i^{(t)}, \theta_m\right)^{\alpha^{(t)}} q_i^{(t)}\left(\theta_m\right)}{f_i\left(y_i^{(t)} | x_i^{(t)}, \theta_*\right)^{\alpha^{(t)}} q_i^{(t)}\left(\theta_*\right)}.
				\end{align}
				
				By taking conditional expectation on both sides of the above equation and noting that $q_i^{(t)}$ is $\mathcal{F}_t$-measurable,  where  $\mathcal{F}_t$ denote the $\sigma$-algebra generated by $\{ (x_i^{(0)},y_i^{(0)}), (x_i^{(1)},y_i^{(1)}), \cdots, (x_i^{(t-1)},y_i^{(t-1)})|i\in\mathcal{N} \}$. Then 
				\begin{align} &\mathbb{E}\Big[\tfrac{1}{N}\sum_{i=1}^{N}\frac{q_i^{(t+1)}\left(\theta_m\right)}{q_i^{(t+1)}\left(\theta_*\right)}|\mathcal{F}_t\Big]
\notag\\ 			&\leq\tfrac{1}{N}\sum_{i=1}^{N}\frac{q_i^{(t)}\left(\theta_m\right)}{q_i^{(t)}\left(\theta_*\right)}\mathbb{E}\Bigg[\left(\frac{f_i\left(y_i^{(t)} | x_i^{(t)}, \theta_m\right)} {f_i\left(y_i^{(t)} | x_i^{(t)}, \theta_*\right) }\right)^{\alpha^{(t)}}|\mathcal{F}_t\Bigg]\notag\\				&\leq\tfrac{1}{N}\sum_{i=1}^{N}\frac{q_i^{(t)}\left(\theta_m\right)}{q_i^{(t)}\left(\theta_*\right)}\mathbb{E}\Bigg[\frac{f_i\left(y_i^{(t)} | x_i^{(t)}, \theta_m\right)} {f_i\left(y_i^{(t)} | x_i^{(t)}, \theta_*\right) }|\mathcal{F}_t\Bigg]^{\alpha^{(t)}}\notag\\	
					&=\tfrac{1}{N}\sum_{i=1}^{N}\tfrac{q_i^{(t)}\left(\theta_m\right)}{q_i^{(t)}\left(\theta_*\right)}\Bigg[\int_{y_i^{(t)}} f_i(y_i^{(t)}|x_i^{(t)},\theta_*)\tfrac{f_i\left(y_i^{(t)} | x_i^{(t)}, \theta_m\right)} {f_i\left(y_i^{(t)} | x_i^{(t)}, \theta_*\right) }dy_i^{(t)}\Bigg]^{\alpha^{(t)}}\notag\\
					&=\tfrac{1}{N}\sum_{i=1}^{N}\frac{q_i^{(t)}\left(\theta_m\right)}{q_i^{(t)}\left(\theta_*\right)},\label{expect_q_q}
				\end{align}
				where the second inequality holds since $x^\alpha,~0<\alpha<1$ is a concave function. Therefore, $\tfrac{1}{N}\sum_{i=1}^{N}\frac{q_i^{(t)}\left(\theta_m\right)}{q_i^{(t)}\left(\theta_*\right)}$ is a non-nenagtive supermartingale.
				Hence by the supermartingale convergence theorem, we conclude  its almost sure convergence,    denoted as $\nu_m$.
				%
			\end{proof}

			In the following, we show that  every agent's   estimated belief of $M$ possible parameters converges to a common belief $\tilde{\boldsymbol{q}}\triangleq(\tilde{q}(\theta_1),\tilde{q}(\theta_2),\cdots,\tilde{q}(\theta_M))^T\in\mathbb{R}^M$.
			\newtheorem{theorem}{Theorem}
			\begin{theorem}\label{lemma_be_con}
				Let Assumptions  \ref{assum_boundbelief}, \ref{assum_graph}, and \ref{assum_step} hold. Then for every agent $i\in\mathcal{N}$, its belief sequence $\{q_i^{(t)}\}_{t\geq0}$ generated by Algorithm \ref{alg:CDSA} converges to a common belief with the form
				\begin{equation}
					\tilde{q}(\theta_m)=\frac{\nu_m}{\sum_{m=1}^{M}\nu_m}  ~\text{for~each~ }m=1,\cdots,M, \label{converq}
				\end{equation}
				where $\nu_m=\lim\limits_{t \rightarrow \infty}\frac{1}{N}\sum_{i=1}^{N}\tfrac{q_i^{(t)}\left(\theta_m\right)}{q_i^{(t)}\left(\theta_*\right)}$ is given in Lemma \ref{lemma_logcon}.
				\end{theorem}	
			\begin{proof}
				Performing an exponential operation on both side of (\ref{con_log}), we have
				\begin{align*}
					&\lim_{t \rightarrow \infty}\tfrac{q_i^{(t+1)}\left(\theta_m\right)}{q_i^{(t+1)}\left(\theta_*\right)} \cdot \tfrac{1}{\exp \left(\tfrac{1}{N} \sum_{i=1}^N \log \tfrac{q_i^{(t+1)}\left(\theta_m\right)}{q_i^{(t+1)}\left(\theta_*\right)}\right)} = 1 \notag\\
					\Rightarrow&\lim_{t \rightarrow \infty}\tfrac{1}{N} \sum_{i=1}^N \tfrac{q_i^{(t+1)}\left(\theta_m\right)}{q_i^{(t+1)}\left(\theta_*\right)} \cdot \tfrac{1}{\exp \left(\tfrac{1}{N} \sum_{i=1}^N \log \tfrac{q_i^{(t+1)}\left(\theta_m\right)}{q_i^{(t+1)}\left(\theta_*\right)}\right)} =1 \notag\\
					\Rightarrow &\lim_{t \rightarrow \infty} \left[\tfrac{1}{N} \sum_{i=1}^N \tfrac{q_i^{(t+1)}\left(\theta_m\right)}{q_i^{(t+1)}\left(\theta_*\right)}-\exp \left(\tfrac{1}{N} \sum_{i=1}^N \log \tfrac{q_i^{(t+1)}\left(\theta_m\right)}{q_i^{(t+1)}\left(\theta_*\right)}\right) \right]=0\notag.
\end{align*}
This together with Lemma~ \ref{lemma_logcon} implies that
\begin{align}
				 &\lim_{t \rightarrow \infty}\tfrac{1}{N} \sum_{i=1}^N \log \tfrac{q_i^{(t+1)}\left(\theta_m\right)}{q_i^{(t+1)}\left(\theta_*\right)} =  \log \nu_m\notag.
				\end{align}
Then by using of Lemma ~\ref{lemma_logcon}, we derive
\begin{align}
				 & \lim_{t \rightarrow \infty}\log \tfrac{q_i^{(t+1)}\left(\theta_m\right)}{q_i^{(t+1)}\left(\theta_*\right)}\xlongequal{}\log\nu_m\label{lim_lognu}.
				\end{align}
				Therefore,  by using Assumption~\ref{assum_graph}, we obtain that
				\begin{align}\label{limit_nu}
					&\lim_{t \rightarrow \infty}\exp\left(\sum_{j=1}^{N}w_{ij} \log \tfrac{q_j^{(t+1)}\left(\theta_m\right)}{q_j^{(t+1)}\left(\theta_*\right)}\right)=\nu_m .
				\end{align}
				
				On the other hand, by the belief update rules in (\ref{q_up}),
				\begin{align}
					q_i^{(t+1)}(\theta_*)&=\tfrac{\exp(\sum_{j\in \mathcal{N}_i}w_{ij}log(b_j^{(t)}(\theta_*)))}{\sum_{\theta\in\Theta}\exp(\sum_{j\in \mathcal{N}_i}w_{ij}log(b_j^{(t)}(\theta)))}\notag
					\\& =\left(1+\sum_{\theta \neq \theta_*} \exp \left( \tfrac{ \sum_{j=1}^N w_{i j} \log b_j^{(t)}(\theta) }{\sum_{j=1}^N w_{i j} \log b_j^{(t)}\left(\theta_*\right)} \right)\right)^{-1} \notag
					\\& =\left(1+\sum_{\theta \neq \theta_*} \exp \left(\sum_{j=1}^N w_{i j} \log \tfrac{b_j^{(t)}(\theta)}{b_j^{(t)}\left(\theta_*\right)}\right)\right)^{-1} \notag\\
					& \overset{(\ref{b_up})}{=}\Bigg(1+\sum_{\theta \neq \theta_*} \exp \Big(\sum_{j=1}^N w_{i j} \alpha^{(t)} \log \tfrac{f_j\left(y_j^{(t)}|x_j^{(t)}, \theta\right)}{f_j\left(y_j^{(t)}|x_j^{(t)}, \theta_*\right)}\notag\\
					&\quad+\sum_{j=1}^N w_{i j} \log \tfrac{q_j^{(t)}(\theta)}{q_j^{(t)}\left(\theta_*\right)}\Big)\Bigg)^{-1}\label{gene_q},~ \forall i\in \mathcal{N},
				\end{align}
where the third equality in the above equation is achieved  similarly to the  third equality of \eqref{log_qq}.
				
				By recalling from  Assumptions \ref{assum_boundbelief} and   \ref{assum_step} that  $\log \tfrac{f_j\left(y_j^{(t)}|x_j^{(t)}, \theta\right)}{f_j\left(y_j^{(t)}|x_j^{(t)}, \theta_*\right)}$ is bounded and $\lim\limits_{t \rightarrow \infty}\alpha^{(t)}=0$. Thus,
				\begin{align}
					\lim_{t \rightarrow \infty}\sum_{j=1}^N w_{i j} \alpha^{(t)} \log \frac{f_j\left(y_j^{(t)}| x_j^{(t)}, \theta\right)}{f_j\left(y_j^{(t)}| x_j^{(t)}, \theta_*\right)}=0.\label{t0}
				\end{align}
				Take $\theta_*=\theta_1$ without loss of generality. Then by substituting (\ref{limit_nu}) and (\ref{t0}) into (\ref{gene_q}), we have
				\begin{align}
					\lim_{t \rightarrow \infty} q_i^{(t+1)}(\theta_*)=(1+\sum_{m=2}^{M}\nu_m)^{-1}, ~a.s.
				\end{align}
				Further, applying (\ref{lim_lognu}) into above relation yields
				\begin{equation}
					\lim_{t \rightarrow \infty} q_i^{(t+1)}(\theta_m)=\frac{\nu_m}{1+\sum_{m=2}^{M}\nu_m},~a.s.~\forall i\in \mathcal{N}
				\end{equation}
				Therefore,  Theorem \ref{lemma_be_con} can be proved by noting that $\nu_1=1$  with the notation $\theta_*=\theta_1$.
			\end{proof}
			
			Though the  above result shows  that every agent's   belief  converges to a common belief, which  does not mean that   the belief vector  is $1$ for the  element  with true parameter $\theta_*.$  Therefore,
			we need  to further prove  its  convergence to a true parameter, i.e. $\tilde{\boldsymbol{q}}\rightarrow \boldsymbol{q}^*$, where
			in vector $\boldsymbol{q}^*(\boldsymbol{\theta})$ only $q(\theta_*)=1$, while  other $q(\theta_m)|_{\theta_m\neq \theta_*}=0$.
This result along with its proof will be given in Theorem \ref{thm3}.

			\subsection{Decision Convergence}
			
			For each $i\in \mathcal{N},$ define
			\begin{align}
				\boldsymbol{q}_i^{(t)}(\boldsymbol{\theta})&\triangleq(q_i^{(t)}(\theta_1),q_i^{(t)}(\theta_2),\cdots,q_i^{(t)}(\theta_M))^T\in\mathbb{R}^M,\label{def_bold_q}\\
				\boldsymbol{J}_i(x,\boldsymbol{\theta})&\triangleq (J_i(x,\theta_1),J_i(x,\theta_2),\cdots,J_i(x,\theta_M))^T\in\mathbb{R}^M. \label{def-boldJi}
			\end{align}
			Then  the  expected cost function  (\ref{be_av})  averaging across  the belief   $q_i^{(t)}$ equals to   $\boldsymbol{q}_i^{(t)}(\boldsymbol{\theta})^T\boldsymbol{J}_i(x_i^{(t)},\boldsymbol{\theta})$, i.e.,
			$	\tilde{J}_i(x_i^{(t)},\boldsymbol{\theta})=\boldsymbol{q}_i^{(t)}(\boldsymbol{\theta})^T\boldsymbol{J}_i(x_i^{(t)},\boldsymbol{\theta})$ . We  re-denote
			$\tilde{J}_i(x_i^{(t)},\boldsymbol{\theta})$ as $F_i(x_i^{(t)},\boldsymbol{q}_i^{(t)})$ to clearly show its dependence on the decision $x_i^{(t)}$ and  the belief  $\boldsymbol{q}_i^{t}$, i.e.,
			\begin{equation}\label{def_F} F_i(x_i^{(t)},\boldsymbol{q}_i^{(t)})\triangleq\boldsymbol{q}_i^{(t)}(\boldsymbol{\theta})^T\boldsymbol{J}_i(x_i^{(t)},\boldsymbol{\theta}).
			\end{equation}
			Therefore, each agent's local cost function can be reformulate as $\boldsymbol{q}^*(\boldsymbol{\theta})^T\boldsymbol{J}_i(x,\boldsymbol{\theta})$, and  the original distributed objective function (\ref{problem}) can be rewritten as
			\begin{align}
				\min_{x\in \mathbb{R}} \frac{1}{N}\sum_{i=1}^{N} \boldsymbol{q}^*(\boldsymbol{\theta})^T\boldsymbol{J}_i(x,\boldsymbol{\theta}) \triangleq \min_{x\in \mathbb{R}} \frac{1}{N}\sum_{i=1}^{N} F_i(x, \boldsymbol{q}^*). \label{newfunc}
			\end{align}
			We denote by  $x^*( \boldsymbol{q})$ the optimal solution to  the optimization problem $\min_{x}  \frac{1}{N}\sum_{i=1}^{N} F_i(x, \boldsymbol{q})$, namely,
			\begin{align}\label{def-xstarq}x^*( {\boldsymbol{q}})={\arg\min}_{x\in \mathbb{R}}~\frac{1}{N}\sum_{i=1}^{N}F_i(x, {\boldsymbol{q}}).
			\end{align}
			Then  $  x^*(\boldsymbol{q}^*)=x_{*}$, which is the optimal solution   to the  problem  \eqref{problem}.
			Besides, step \eqref{decision_update} in Algorithm \ref{alg:CDSA} can be reformulated as
			\begin{align}
				x_i^{(t+1)}=\sum_{j=1}^{N}w_{ij} \left[x_j^{(t)}-\alpha^{(t)}\nabla_x F_j(x_j^{(t)},\boldsymbol{q}_j^{(t)})\right].\label{newx_iter}
			\end{align}
			
			In the following, we will show  that  the decision sequence $\{x_i^{(t)}\}_{t\geq0}$ for every agent $i$ converges to a common  solution
			$x^*(\tilde{\boldsymbol{q}})$ (convergence to the true optimal solution  $  x^*(\boldsymbol{q}^*)$ will be presented in later part), where $\tilde{\boldsymbol{q}}$ is given in Theorem \ref{lemma_be_con} .
			
			First of all,  the properties of the newly shaped function $F_i(x,\boldsymbol{q}_i)$ defined by \eqref{def_F} are summarized below, which can be obtained directly from \cite[Section 3.2.1]{cvbook} for the strongly convexity property and \cite{lips_smooth} for the Lipschitz smooth property.
			
			\begin{lemma}\label{lemma_allproperty}
				Let Assumption \ref{assum_func} hold. Then for all $i\in\mathcal{N}$ and for all $\boldsymbol{q}_i\in\mathbb{R}^M$, $F_i(x,\boldsymbol{q}_i)$   is strongly convex and Lipschitz smooth in $x$ with constant $\mu$ and $L$.
			\end{lemma}

	In the following, we will show the recursions on the optimization error $\|\bar{x}^{(t+1)}-x^*(\tilde{\boldsymbol{q}})\|$ in Lemma \ref{lemma_op}, and consensus error  $\|\boldsymbol{x}^{(t+1)}-\boldsymbol{1}\bar{x}^{(t+1)} \|$ in Lemma \ref{lemma_consens}. For the sake of simplicity, we give  some  more notations below.
		\begin{align}
			\boldsymbol{x}^{(t)}&\triangleq(x_1^{(t)},x_2^{(t)},\cdots,x_N^{(t)})^T\in\mathbb{R}^N, \label{mathbbx}\\
			\bar{x}^{(t)}&\triangleq\frac{1}{N}\sum_{i=1}^{N}x_i^{(t)}\in\mathbb{R}\label{barx},\\
			\boldsymbol{Q}^{(t)}&\triangleq (\boldsymbol{q}_1^{(t)}, \boldsymbol{q}_2^{(t)},\cdots,\boldsymbol{q}_N^{(t)})^T\in\mathbb{R}^{N\times M},\label{big_Q}\\
			 \frac{1}{N}\sum_{i=1}^{N} &F_i(x_i^{(t)}, \boldsymbol{q}_i^{(t)})\triangleq  \bar{F}(\boldsymbol{x}^{(t)}, \boldsymbol{Q}^{(t)})\in\mathbb{R}\label{full_func},\\ \mathbb{F}(\boldsymbol{x}^{(t)},\boldsymbol{Q}^{(t)})&\triangleq\left(F_1\left(x_1^{(t)},\boldsymbol{q}_1^{(t)}\right),F_2\left(x_2^{(t)},\boldsymbol{q}_2^{(t)}\right),\right.\notag\\
			&\qquad\left.\cdots,F_N\left(x_N^{(t)},\boldsymbol{q}_N^{(t)}\right)\right)^T\in\mathbb{R}^N\label{mathbbF}
		\end{align}
	
	\begin{lemma}\label{lemma_op}
		Let Assumptions \ref{assum_graph}, \ref{assum_step}, and \ref{assum_func} hold. Under Algorithm \ref{alg:CDSA}, supposing stepsize $\alpha^{(t)}<\frac{1}{2L}$, we can bound the gap between $\bar{x}^{(t+1)}$ and $x^*(\tilde{\boldsymbol{q}})$ as follows,
		\begin{align}
			\|&\bar{x}^{(t+1)}-x^*(\tilde{\boldsymbol{q}})\|\notag\\
			&\leq\sqrt{1-\alpha^{(t)}\mu(1-2L\alpha^{(t)})}\|\bar{x}^{(t)}-x^*(\tilde{\boldsymbol{q}})\|\notag\\
			&~+\frac{ [\alpha^{(t)}]^{0.5}L}{\sqrt{\mu N}} \|\boldsymbol{x}^{(t)}-\boldsymbol{1}\bar{x}^{(t)}\|+\frac{\sqrt{2}L\alpha^{(t)}}{\sqrt{N}} \|\boldsymbol{x}^{(t)}-\boldsymbol{1}\bar{x}^{(t)}\|\notag\\
			&~+\alpha^{(t)}\frac{1}{N}\sum_{i=1}^{N}\|\boldsymbol{q}_i^{(t)}-\tilde{\boldsymbol{q}} \| \| \nabla_x\boldsymbol{J}_i(x^*(\tilde{\boldsymbol{q}}),\boldsymbol{\theta})\|
		\end{align}
	 \end{lemma}
		
		\begin{proof}
			By using the optimality condition of the unconstrained optimization problem  \eqref{newfunc}, we have $ \frac{1}{N}\sum_{i=1}^{N} \nabla_xF_i(x^*(\tilde{\boldsymbol{q}}), \tilde{\boldsymbol{q}})=\nabla_x\bar{F}(\boldsymbol{1}x^*(\tilde{\boldsymbol{q}}),\boldsymbol{1}\otimes\tilde{\boldsymbol{q}}^T)=0.$ Then by using   iteration of $x_i^{(t)}$ in \eqref{newx_iter}, and the definition of $\bar{x}^{(t)}$ and $\bar{F}$ in \eqref{barx} and \eqref{full_func}, we have
			\begin{align}
				&\|\bar{x}^{(t+1)}\!-\! x^*(\tilde{\boldsymbol{q}})\|\!=\!\Big\|\frac{1}{N}\!\sum_{i=1}^{N}\!\sum_{j=1}^{N}\!w_{ij}\big[x_j^{(t)}\!-\!\alpha^{(t)}\nabla_xF_j(x_j^{(t)},\boldsymbol{q}_j^{(t)})\big]\notag\\
				&\qquad-\big[x^*(\tilde{\boldsymbol{q}})-\alpha^{(t)}\nabla_x \bar{F}(\boldsymbol{1}x^*(\tilde{\boldsymbol{q}}),\boldsymbol{1}\otimes\tilde{\boldsymbol{q}}^T)\big]\Big\|\notag\\\notag
			&= \Big\|\ \bar{x}^{(t)}-\alpha^{(t)}\nabla_x \bar{F}(\boldsymbol{x}^{(t)},\boldsymbol{Q}^{(t)})\\\notag
			&\qquad-x^*(\tilde{\boldsymbol{q}})+\alpha^{(t)}\nabla_x \bar{F}(\boldsymbol{1}x^*(\tilde{\boldsymbol{q}}),\boldsymbol{1}\otimes\tilde{\boldsymbol{q}}^T)\Big\|\\\notag
			&\leq\! \big\|\bar{x}^{(t)}- x^*(\tilde{\boldsymbol{q}})\\ \notag		
			&\quad-\alpha^{(t)}\left( \nabla_x \bar{F}(\boldsymbol{x}^{(t)},\boldsymbol{Q}^{(t)})- \nabla_x \bar{F}(\boldsymbol{1}x^*(\tilde{\boldsymbol{q}}),\boldsymbol{Q}^{(t)})\!\right)\notag\\
			&\quad-\alpha^{(t)}\left( \nabla_x\bar{F}(\boldsymbol{1}x^*(\tilde{\boldsymbol{q}}),\boldsymbol{Q}^{(t)})-\nabla_x\bar{F}(\boldsymbol{1}x^*(\tilde{\boldsymbol{q}}),\boldsymbol{1}\otimes\tilde{\boldsymbol{q}}^T)\right)\big\|\notag\\
			&\leq\! \big\|\bar{x}^{(t)}- x^*(\tilde{\boldsymbol{q}}) \label{barxx}\\		
			&~~\qquad-\alpha^{(t)}\left( \nabla_x \bar{F}(\boldsymbol{x}^{(t)},\boldsymbol{Q}^{(t)})- \nabla_x \bar{F}(\boldsymbol{1}x^*(\tilde{\boldsymbol{q}}),\boldsymbol{Q}^{(t)})\!\right)\big\|\notag\\
			&\quad+\alpha^{(t)}\big\| \nabla_x\bar{F}(\boldsymbol{1}x^*(\tilde{\boldsymbol{q}}),\boldsymbol{Q}^{(t)})
\-\nabla_x\bar{F}(\boldsymbol{1}x^*(\tilde{\boldsymbol{q}}),\boldsymbol{1}\otimes\tilde{\boldsymbol{q}}^T)\big\|,\notag
			\end{align}
		where the second equality holds by using $\frac{1}{N}\!\sum_{i=1}^{N}\! \!w_{ij}=1$, and the last equality utilizes the triangle inequality.
		
		The first term in the right-hand side of \eqref{barxx} can be further bounded by first writing the following expansion:
		\begin{align}
			&\big\|\bar{x}^{(t)}- x^*(\tilde{\boldsymbol{q}}) \notag\\		
			&\qquad-\alpha^{(t)}\left( \nabla_x \bar{F}(\boldsymbol{x}^{(t)},\boldsymbol{Q}^{(t)})- \nabla_x \bar{F}(\boldsymbol{1}x^*(\tilde{\boldsymbol{q}}),\boldsymbol{Q}^{(t)})\!\right)\big\|^2\notag\\
			&=\|\bar{x}^{(t)}- x^*(\tilde{\boldsymbol{q}})\|^2\notag\\
			&~-2\alpha^{(t)}\!(\bar{x}^{(t)}\!-\! x^*(\tilde{\boldsymbol{q}}))^T\!\!\left(\!\nabla_x \bar{F}(\boldsymbol{x}^{(t)},\!\boldsymbol{Q}^{(t)}\!)\!-\!\nabla_x \bar{F}(\boldsymbol{1}x^*(\tilde{\boldsymbol{q}}),\boldsymbol{Q}^{(t)}\!)\!\right)\notag\\
			&~+[\alpha^{(t)}]^2\big\|\nabla_x \bar{F}(\boldsymbol{x}^{(t)},\boldsymbol{Q}^{(t)})- \nabla_x \bar{F}(\boldsymbol{1}x^*(\tilde{\boldsymbol{q}}),\boldsymbol{Q}^{(t)})\big\|^2\notag\\
			&\leq \|\bar{x}^{(t)}- x^*(\tilde{\boldsymbol{q}})\|^2 \label{termf1}\\
			&~\underbrace{-2\alpha^{(t)}(\bar{x}^{(t)}-\! x^*(\tilde{\boldsymbol{q}}))^T(\nabla_x\bar{F}(\boldsymbol{x}^{(t)},\boldsymbol{Q}^{(t)})-\!\nabla_x\bar{F}(\boldsymbol{1}\bar{x}^{(t)},\boldsymbol{Q}^{(t)}))}_{Term~1}\notag\\
			&~\underbrace{+2[\alpha^{(t)}]^2\|\nabla_x\bar{F}(\boldsymbol{x}^{(t)},\boldsymbol{Q}^{(t)})-\nabla_x\bar{F}(\boldsymbol{1}\bar{x}^{(t)},\boldsymbol{Q}^{(t)})\|^2}_{Term~2}\notag\\
			&~\underbrace{\begin{aligned}
					+2[\alpha^{(t)}]^2\|\nabla_x\bar{F}(\boldsymbol{1}\bar{x}^{(t)},\boldsymbol{Q}^{(t)})-\nabla_x\bar{F}(\boldsymbol{1}x^*(\tilde{\boldsymbol{q}}),\boldsymbol{Q}^{(t)})\|^2\\
					-2\alpha^{(t)}\!(\bar{x}^{(t)}\!-\! x^*(\tilde{\boldsymbol{q}}))^T\!(\nabla_x\bar{F}(\boldsymbol{1}\bar{x}^{(t)},\!\boldsymbol{Q}^{(t)})\!-\!\nabla_x\bar{F}\!(\boldsymbol{1}x^*(\tilde{\boldsymbol{q}}),\boldsymbol{Q}^{(t)}))
				\end{aligned}}_{Term~3}\notag
		\end{align}
	(Plus $Term~3$ contains two terms) where the last equality is obtained by adding and subtracting the same terms and together with  $(a+b)^2\leq 2a^2+2b^2$.

	Recalling from the definition of $\bar{F}$ in \eqref{full_func}  and together with the triangle  equality $\|\sum_{i=1}^{N}z_i\|\leq \sum_{i=1}^{N}\|z_i\|$, we have
	\begin{align} &\|\nabla_x\bar{F}(\boldsymbol{x}^{(t)},\boldsymbol{Q}^{(t)})-\!\nabla_x\bar{F}(\boldsymbol{1}\bar{x}^{(t)},\boldsymbol{Q}^{(t)})\|\notag\\ &=\big\|\frac{1}{N}\sum_{i=1}^{N}\left(F_i(x_i^{(t)},\boldsymbol{q}_i^{(t)})-F_i(\bar{x}^{(t)},\boldsymbol{q}_i^{(t)})\right)\big\|\notag\\
		&\leq \frac{1}{N}\sum_{i=1}^{N} \|F_i(x_i^{(t)},\boldsymbol{q}_i^{(t)})-F_i(\bar{x}^{(t)},\boldsymbol{q}_i^{(t)})\|\notag\\
		&\leq \frac{1}{N}\sum_{i=1}^{N} L\|x_i^{(t)}-\bar{x}^{(t)}\| \label{barFL},
	\end{align}
where the last inequality uses the Lipschitz smoothness of $F_i(x,\boldsymbol{q}_i)$ to $x$ in Lemma \ref{lemma_allproperty}. Therefore, based on the Cauchy-Schwartz inequality, we can bound Term 1 as follows
 \begin{align}
 	&Term~1 \leq 2\alpha^{(t)}\|\bar{x}^{(t)}-\! x^*(\tilde{\boldsymbol{q}})\|\notag\\
 	&\qquad\times\|\nabla_x\bar{F}(\boldsymbol{x}^{(t)},\boldsymbol{Q}^{(t)})-\!\nabla_x\bar{F}(\boldsymbol{1}\bar{x}^{(t)},\boldsymbol{Q}^{(t)})\|\notag\\
 	&=2\! \left(\![\alpha^{(t)}]^{0.5}\mu^{0.5}\|\bar{x}^{(t)}-x^*(\tilde{\boldsymbol{q}})\|\!\right)\! \Big(\!\tfrac{[\alpha^{(t)}]^{0.5}L}{\mu^{0.5}N}\!\sum_{i=1}^{N}\!\|x_i^{(t)}\!-\bar{x}^{(t)}\|\!\Big)\notag\\
 	&\leq \alpha^{(t)}\mu\|\bar{x}^{(t)}-x^*(\tilde{\boldsymbol{q}})\|^2+\frac{\alpha^{(t)}L^2}{\mu N^2} \times N\sum_{i=1}^{N}\|x_i^{(t)}-\bar{x}^{(t)}\|^2\notag\\
 	&=\alpha^{(t)}\mu\|\bar{x}^{(t)}-x^*(\tilde{\boldsymbol{q}})\|^2+\frac{\alpha^{(t)}L^2}{\mu N} \|\boldsymbol{x}^{(t)}-\boldsymbol{1}\bar{x}^{(t)}\|^2\label{term1_1},
 \end{align}
where the penultimate inequality is followed by $2ab\leq a^2+b^2$ for all $a,b>0$ and $(\sum_{i=1}^{N}\|z_i\|)^2\leq N\sum_{i=1}^{N}\|z_i\|^2$.

As for Term 2,  by using \eqref{barFL}, we achieve
\begin{align}
	Term~2&\leq 2[\alpha^{(t)}]^2  \left(\frac{L}{N}\sum_{i=1}^{N}\|x_i^{(t)}-\bar{x}^{(t)}\|\right)^2\notag\\
	&\leq 2[\alpha^{(t)}]^2 \frac{L^2}{N}\sum_{i=1}^{N}\|x_i^{(t)}-\bar{x}^{(t)}\|^2\notag\\
	&=\frac{2L^2[\alpha^{(t)}]^2}{N} \|\boldsymbol{x}^{(t)}-\boldsymbol{1}\bar{x}^{(t)}\|^2\label{term1_2}.
\end{align}

Recalling the definition of $\bar{F}$ in \eqref{full_func} and the Lipschitz smooth property of $F_i$ in Lemma \ref{lemma_allproperty}, we have
\begin{align}\label{lipsF} &\|\nabla_x\bar{F}(\boldsymbol{1}\bar{x}^{(t)},\boldsymbol{Q}^{(t)})-\nabla_x\bar{F}(\boldsymbol{1}x^*(\tilde{\boldsymbol{q}}),\boldsymbol{Q}^{(t)})\|^2\notag\\ &=\|\frac{1}{N}\sum_{i=1}^{N}\left(\nabla_xF_i(\bar{x}^{(t)},\boldsymbol{q}_i^{(t)})-\nabla_xF_i(x^*(\tilde{\boldsymbol{q}}),\boldsymbol{q}_i^{(t)})\right)\|^2\notag\\
	&\leq \frac{1}{N}\sum_{i=1}^{N} \|\nabla_xF_i(\bar{x}^{(t)},\boldsymbol{q}_i^{(t)})-\nabla_xF_i(x^*(\tilde{\boldsymbol{q}}),\boldsymbol{q}_i^{(t)})\|^2\\
	&\leq\!\tfrac{L}{N}\!\sum_{i=1}^{N} \! (\bar{x}^{(t)}\!-\! x^*(\tilde{\boldsymbol{q}}))^T\!(\nabla_xF_i(\bar{x}^{(t)},\!\boldsymbol{q}_i^{(t)})\!
-\!\nabla_xF_i(x^*(\tilde{\boldsymbol{q}}),\boldsymbol{q}_i^{(t)}\!)\!) \notag
\end{align}
where the last inequality is followed by the Lipschitz smooth properties \cite[Equation (2.1.8)]{nesterov2013introductory}.

In addition, based on the strong convexity of $F_i$ in Lemma \ref{lemma_allproperty}, we have
\begin{align}
	(\bar{x}^{(t)}- &x^*(\tilde{\boldsymbol{q}}))^T\!(\nabla_xF_i(\bar{x}^{(t)},\!\boldsymbol{q}_i^{(t)})\!-\!\nabla_xF_i(x^*(\tilde{\boldsymbol{q}}),\boldsymbol{q}_i^{(t)}\!)\!)\notag\\
	&\geq \mu \|\bar{x}^{(t)}-x^*(\tilde{\boldsymbol{q}})\|^2\label{convexF}
\end{align}

By recalling the definition of $\bar{F}$ in \eqref{full_func}  and using
  \eqref{lipsF}, we can further bound Term 3 as follows
\begin{align}
	&Term~3\leq \notag\\
	&\!\tfrac{2 [\alpha^{(t)}]^2\! L}{N}\!\sum_{i=1}^{N} \! (\bar{x}^{(t)}\!-\! x^*(\tilde{\boldsymbol{q}}))^T\!(\nabla_xF_i(\bar{x}^{(t)},\!\boldsymbol{q}_i^{(t)})\!-\!\nabla_xF_i(x^*(\tilde{\boldsymbol{q}}),\boldsymbol{q}_i^{(t)}\!)\!)\notag\\
	&\!-\!\tfrac{2\alpha^{(t)}}{N}\! (\bar{x}^{(t)}\!-\! x^*(\tilde{\boldsymbol{q}}))^T\!\sum_{i=1}^{N}(\nabla_xF_i(\bar{x}^{(t)},\!\boldsymbol{q}_i^{(t)})\!-\!\nabla_xF_i(x^*(\tilde{\boldsymbol{q}}),\boldsymbol{q}_i^{(t)}\!)\!)\notag\\
	&=-\tfrac{2\alpha^{(t)}}{N}(1-\alpha^{(t)}L)\notag\\
	&~\times \sum_{i=1}^{N} \! (\bar{x}^{(t)}\!-\! x^*(\tilde{\boldsymbol{q}}))^T\!(\nabla_xF_i(\bar{x}^{(t)},\!\boldsymbol{q}_i^{(t)})\!-\!\nabla_xF_i(x^*(\tilde{\boldsymbol{q}}),\boldsymbol{q}_i^{(t)}\!))\notag\\
	& \leq  -\tfrac{2\alpha^{(t)}}{N}(1-\alpha^{(t)}L) \sum_{i=1}^{N}\mu\|\bar{x}^{(t)}-x^*(\tilde{\boldsymbol{q}})\|^2\notag\\
	&=-2\alpha^{(t)}(1-\alpha^{(t)}L)\mu\|\bar{x}^{(t)}-x^*(\tilde{\boldsymbol{q}})\|^2\label{term1_3},
\end{align}
where the last inequality holds by using  \eqref{convexF} and $ -\tfrac{2\alpha^{(t)}}{N}(1-\alpha^{(t)}L) <0$ since  $\alpha^{(t)}<\frac{1}{2L}.$

Then by substituting \eqref{term1_1}, \eqref{term1_2}, and \eqref{term1_3} into \eqref{termf1}, we get
\begin{align}
	&\big\|\bar{x}^{(t)}- x^*(\tilde{\boldsymbol{q}}) \notag\\		
	&\qquad-\alpha^{(t)}\left( \nabla_x \bar{F}(\boldsymbol{x}^{(t)},\boldsymbol{Q}^{(t)})- \nabla_x \bar{F}(\boldsymbol{1}x^*(\tilde{\boldsymbol{q}}),\boldsymbol{Q}^{(t)})\!\right)\big\|^2\notag\\
	&\leq (1-\alpha^{(t)}\mu+2[\alpha^{(t)}]^2\mu L)\|\bar{x}^{(t)}-x^*(\tilde{\boldsymbol{q}})\|^2\notag\\
	&~+\frac{\alpha^{(t)}L^2}{\mu N} \|\boldsymbol{x}^{(t)}-\boldsymbol{1}\bar{x}^{(t)}\|^2+\frac{2L^2[\alpha^{(t)}]^2}{N} \|\boldsymbol{x}^{(t)}-\boldsymbol{1}\bar{x}^{(t)}\|^2\notag.
\end{align}

Since $\sqrt{a+b}\leq \sqrt{a}+\sqrt{b}$ for all $a,b\geq0$, the first term on the right hand side of \eqref{barxx} can be bounded by
\begin{align}
	&\big\|\bar{x}^{(t)}- x^*(\tilde{\boldsymbol{q}}) \notag\\		
	&\qquad-\alpha^{(t)}\left( \nabla_x \bar{F}(\boldsymbol{x}^{(t)},\boldsymbol{Q}^{(t)})- \nabla_x \bar{F}(\boldsymbol{1}x^*(\tilde{\boldsymbol{q}}),\boldsymbol{Q}^{(t)})\!\right)\big\|\notag\\
	&\leq \sqrt{1-\alpha^{(t)}\mu+2[\alpha^{(t)}]^2\mu L}\|\bar{x}^{(t)}-x^*(\tilde{\boldsymbol{q}})\|\label{termf11}\\
	&~+\frac{ [\alpha^{(t)}]^{0.5}L}{\sqrt{\mu N}} \|\boldsymbol{x}^{(t)}-\boldsymbol{1}\bar{x}^{(t)}\|+\frac{\sqrt{2}L\alpha^{(t)}}{\sqrt{N}} \|\boldsymbol{x}^{(t)}-\boldsymbol{1}\bar{x}^{(t)}\|\notag
\end{align}

Consider the second term on the right-hand side of \eqref{barxx}. Recalling the definition of newly shaped function in \eqref{newfunc} and \eqref{full_func}, we have
\begin{align}
	&\alpha^{(t)}\big\| \nabla_x\bar{F}(\boldsymbol{1}x^*(\tilde{\boldsymbol{q}}),\boldsymbol{Q}^{(t)})-\nabla_x\bar{F}(\boldsymbol{1}x^*(\tilde{\boldsymbol{q}}),\boldsymbol{1}\otimes\tilde{\boldsymbol{q}}^T)\big\|\notag\\
	&=\alpha^{(t)}\big\| \frac{1}{N}\sum_{i=1}^{N}\left(\nabla_xF_i(x^*(\tilde{\boldsymbol{q}}),\boldsymbol{q}_i^{(t)})-\nabla_xF_i (x^*(\tilde{\boldsymbol{q}}), \tilde{\boldsymbol{q}})\right)\big\|\notag\\
	&=\alpha^{(t)}\big\|\frac{1}{N}\sum_{i=1}^{N}\left(\boldsymbol{q}_i^{(t)}-\tilde{\boldsymbol{q}}\right)^T\nabla_x\boldsymbol{J}_i(x^*(\tilde{\boldsymbol{q}}),\boldsymbol{\theta})\big\|\notag\\
&\leq \alpha^{(t)}\frac{1}{N}\sum_{i=1}^{N}\|\boldsymbol{q}_i^{(t)}-\tilde{\boldsymbol{q}} \| \| \nabla_x\boldsymbol{J}_i(x^*(\tilde{\boldsymbol{q}}),\boldsymbol{\theta})\|\label{termf2}.
\end{align}
Substituting \eqref{termf11} and \eqref{termf2} into \eqref{barxx} yields the lemma.
		\end{proof}

			In the following lemma, we establish the   recursion for the   consensus error $\left\|\boldsymbol{x}^{(t+1)}-\boldsymbol{1} \bar{x}^{(t+1)}\right\|^2$.

			\begin{lemma}\label{lemma_consens}
				Let Assumptions  \ref{assum_graph} and  \ref{assum_func} hold.  We  then have
				\begin{align}\label{consen_error}
					&	\left\|\boldsymbol{x}^{(t+1)}-\right.\left.\boldsymbol{1} \bar{x}^{(t+1)}\right\|^2 \leqslant \tfrac{3+\rho_w^2}{4}\left\|\boldsymbol{x}^{(t)}-\boldsymbol{1} \bar{x}^{(t)}\right\|^2\\
					&+\tfrac{3 \rho_w^2\left[\alpha^{(t)}\right]^2}{1-\rho_w^2} \!\!\Bigg[2 M^2\! L^2	\|\boldsymbol{x}^{(t)}-\boldsymbol{1}x^*(\tilde{\boldsymbol{q}})\|^2\!\notag\\
					&\qquad \qquad \quad+2 M\sum_{i=1}^N\|\nabla_x \boldsymbol{J}_i(x^*(\tilde{\boldsymbol{q}}), \boldsymbol{\theta})\|^2\Bigg]\!,\notag
				\end{align}
				where $\rho_w$ is the spectral radius of  $W-\frac{\boldsymbol{1}\boldsymbol{1}^{\top}}{N}$.
			\end{lemma}
			\begin{proof}
				By  recalling the definitions of $\bar{x}^{(t)}$ and $\bar{ F}(\boldsymbol{x}^{(t)},\boldsymbol{Q}^{(t)}) $ in (\ref{barx}) and (\ref{full_func}),  together with the double stochasticity of $W$ in Assumption \ref{assum_graph}, we have
				\begin{align}
x_i^{(t+1)}-\bar{x}^{(t+1)}& \overset{\eqref{newx_iter}}{=}   \sum_{j=1}^{N}w_{ij}(x_j^{(t)}-\alpha^{(t)} \nabla_xF_i(x_i^{(t)},\boldsymbol{q}_i^{(t)}))\notag\\
					& -\left(\bar{x}^{(t)}-\alpha^{(t)}\nabla_x\bar{ F}(\boldsymbol{x}^{(t)},\boldsymbol{Q}^{(t)})\right) .\label{46}
				\end{align}
				As a result, consider the vector form. By recalling the definitions of  $\mathbb{F}\left(\boldsymbol{x}^{(t)}, \boldsymbol{Q}^{(t)}\right)$ in (\ref{mathbbF}), we have
				\begin{align*}
					&\|\boldsymbol{x}^{(t+1)}-\boldsymbol{1} \bar{x}^{(t+1)}\|\leq \Big\|W\left(\boldsymbol{x}^{(t)}-\alpha^{(t)} \nabla_x \mathbb{F}\left(\boldsymbol{x}^{(t)}, \boldsymbol{Q}^{(t)}\right)\right)\notag\\
					&-\boldsymbol{1}\left(\bar{x}^{(t)}-\alpha^{(t)}\nabla_x\bar{ F}\left(\boldsymbol{x}^{(t)}, \boldsymbol{Q}^{(t)}\right)\right)\Big\| \notag \\
					& =\Bigg\|\left(W-\frac{\boldsymbol{1}\boldsymbol{1}^{\top}}{N}\right)\Bigg[ \left(\boldsymbol{x}^{(t)}-\boldsymbol{1} \bar{x}^{(t)}\right)\notag\\
					&\quad-\alpha^{(t)} \left(\mathbb{F}\left(\boldsymbol{x}^{(t)}, \boldsymbol{Q}^{(t)}\right)-\boldsymbol{1}\nabla_x \bar{F}\left(\boldsymbol{x}^{(t)}, \boldsymbol{Q}^{(t)}\right) \right)\Bigg]\Bigg\|\notag\\
					&=\Bigg\|\left(W-\frac{\boldsymbol{1}\boldsymbol{1}^{\top}}{N}\right) \left(\boldsymbol{x}^{(t)}-\boldsymbol{1} \bar{x}^{(t)}\right)\notag\\
					&\quad-\alpha^{(t)} \left(W-\frac{\boldsymbol{1}\boldsymbol{1}^{\top}}{N}\right)\left(I-\frac{\boldsymbol{1}\boldsymbol{1}^{\top}}{N}\right)\mathbb{F}\left(\boldsymbol{x}^{(t)}, \boldsymbol{Q}^{(t)}\right)\Bigg\|,
				\end{align*}
			where the second equality holds since $\frac{\boldsymbol{1}\boldsymbol{1}^T}{N}(\boldsymbol{x}^{(t)}-\boldsymbol{1}\bar{x}^{(t)})=\boldsymbol{1}\bar{x}^{(t)}-\boldsymbol{1}\bar{x}^{(t)}=0$, whereas the last equality follows by $\nabla_x \bar{ F}\left(\boldsymbol{x}^{(t)}, \boldsymbol{Q}^{(t)}\right)=\tfrac{\boldsymbol{1}^{\top}}{N}\mathbb{F}\left(\boldsymbol{x}^{(t)}, \boldsymbol{Q}^{(t)}\right)$.
			
			Noticing that $\|I-\frac{\boldsymbol{1}\boldsymbol{1}^{\top}}{N}\|\leq 1$ and $\rho_w$ is the spectral norm of  $\|W-\frac{\boldsymbol{1}\boldsymbol{1}^{\top}}{N}\|$, based on above relation we derive
				\begin{align}
					\|\boldsymbol{x}^{(t+1)}-&\boldsymbol{1} \bar{x}^{(t+1)}\| \notag\\
					& \leq\rho_w \left\|\boldsymbol{x}^{(t)}-\boldsymbol{1} \bar{x}^{(t)}\right\| + \alpha^{(t)}\rho_w \left \|  \nabla_x \mathbb{F}\left(\boldsymbol{x}^{(t)}, \boldsymbol{Q}^{(t)}\right) \right\|.  \notag
				\end{align}
				Hence by using $(a+b)^2\leq a^2+b^2+2ab \leq a^2+b^2+a^2/c +b^2c  $ for any $c >0,$ we obtain that for any $c_1>0$,
				\begin{align}
					&\left\|\boldsymbol{x}^{(t+1)}-\boldsymbol{1} \bar{x}^{(t+1)}\right\|^2  \notag\\
					&\leq \rho_w^2 (1+c_1)\left\|\boldsymbol{x}^{(t)}-\boldsymbol{1} \bar{x}^{(t)}\right\|^2\notag\\
					&\quad+\left[\alpha^{(t)}\right]^2\rho_w^2 (1+\tfrac{1}{c_1})\left\| \nabla_x \mathbb{F}\left(\boldsymbol{x}^{(t)}, \boldsymbol{Q}^{(t)}\right)\right\|^2 . \label{con_error}
				\end{align}
					Note that  for any probability vector $\boldsymbol{q}\in \mathbb{R}^M$, since every element of $\boldsymbol{q}$ is nonnegative and  less than $1$,  we have
				\begin{align}\label{bd-pro}
					\|\boldsymbol{q}\| \leq \sqrt{M}.
				\end{align}
				
				By using (\ref{def_F}) and (\ref{mathbbF}), we can obtain that
				\begin{align}\label{MJ}
					& \left\|\nabla_x \mathbb{F}\left(\boldsymbol{x}^{(t)}, \boldsymbol{Q}^{(t)}\right)\right\|^2 =\sum_{i=1}^N \| \boldsymbol{q}_i^{(t)}(\boldsymbol{\theta})^T\nabla_x\boldsymbol{J}_i(x_i^{(t)},\boldsymbol{\theta})\|^2 \notag
					\\& \leq \sum_{i=1}^N \| \boldsymbol{q}_i^{(t)}(\boldsymbol{\theta})\|^2 \| \nabla_x\boldsymbol{J}_i(x_i^{(t)},\boldsymbol{\theta})\|^2\notag\\
					&\overset{\eqref{bd-pro}}{\leq } M \sum_{i=1}^N   \|\nabla_x \boldsymbol{J}_i(x_i^{(t)},\boldsymbol{\theta})\|^2   .
				\end{align}
				In addition, recalling  the Lipschitz smooth property  in Assumption \ref{assum_func}, and the definition of $\boldsymbol{J}_i(x,\boldsymbol{\theta})$ in  \eqref{def-boldJi}, we  obtain
				\begin{align}\label{bd-boldJi}
					&\|\nabla_x \boldsymbol{J}_i(x_i^{(t)},\boldsymbol{\theta})\|\notag
					\\& = \|\nabla_x \boldsymbol{J}_i   \left(x_i^{(t)}, \boldsymbol{\theta}\right)-\nabla_x \boldsymbol{J}_i(x^*(\tilde{\boldsymbol{q}}), \boldsymbol{\theta})+\nabla_x \boldsymbol{J}_i(x^*(\tilde{\boldsymbol{q}}), \boldsymbol{\theta})\| \notag
					\\& \leq \|\nabla_x \boldsymbol{J}_i(x^*(\tilde{\boldsymbol{q}}), \boldsymbol{\theta})\| \notag \\& +\sqrt{\sum_{m=1}^M \|   \nabla_x J_i   \left(x_i^{(t)}, \theta_m \right)-\nabla_x J_i(x^*(\tilde{\boldsymbol{q}}),  \theta_m) \|^2} \notag
					\\&=\|\nabla_x \boldsymbol{J}_i(x^*(\tilde{\boldsymbol{q}}), \boldsymbol{\theta})\|+\sqrt{M}L\|x_i^{(t)}-x^*(\tilde{\boldsymbol{q}})\|
				\end{align}

					Whereas by using $(a+b)^2\leq 2(a^2+b^2),$ we have
					\begin{align*}
						&\|\nabla_x \boldsymbol{J}_i(x_i^{(t)},\boldsymbol{\theta})\|^2
						\leq \\
						&2 \|\nabla_x \boldsymbol{J}_i(x^*(\tilde{\boldsymbol{q}}), \boldsymbol{\theta})\|^2+2ML^2\|x_i^{(t)}-x^*(\tilde{\boldsymbol{q}})\|^2 .
					\end{align*}
					This together with \eqref{MJ}  produces
				\begin{align}
					& \left\|\nabla_x \mathbb{F}\left(\boldsymbol{x}^{(t)}, \boldsymbol{Q}^{(t)}\right)\right\|^2 \label{vec_F_bound}
					\\&   \leq M  \sum_{i=1}^N    ( 2\|\nabla_x \boldsymbol{J}_i(x^*(\tilde{\boldsymbol{q}}), \boldsymbol{\theta})\|^2+2ML^2\|x_i^{(t)}-x^*(\tilde{\boldsymbol{q}})\|^2).\notag
				\end{align}

				By combining (\ref{con_error}) with (\ref{vec_F_bound}),  and  letting $c_1=\frac{1-\rho_w^2}{2}$, we have
				\begin{align*}
					& \frac{1}{\rho_w^2}\left\|\boldsymbol{x}^{(t+1)}-\boldsymbol{1}\bar{x}^{(t+1)}\right\|^2 \leqslant \frac{3-\rho_w^2}{2}\left\|\boldsymbol{x}^{(t)}-1 \bar{x}^{(t)}\right\|^2+\tfrac{3\left[\alpha^{(t)}\right]^2 }{1-\rho_w^2}\notag\\
					&\times
				\Big[2 M^2 L^2	\|\boldsymbol{x}^{(t)}-\boldsymbol{1}x^*(\tilde{\boldsymbol{q}})\|^2 +2 M \sum_{i=1}^N\|\nabla_x \boldsymbol{J}_i(x^*(\tilde{\boldsymbol{q}}), \boldsymbol{\theta})\|^2   \Big].
				\end{align*}
				Note that $\rho_w^2\left(\frac{3-\rho_w^2}{2}\right) \leq \frac{3+\rho_w^2}{4} \text { by } \rho_w \in(0,1)$. Then multiplying $\rho_w$ on both side of above relation leads to  (\ref{consen_error}).
			\end{proof}
			
			 From now on, we consider the stepsize $\alpha^{(t)}$ of order $\mathcal{O}(\frac{1}{t})$, which also satisfy the Assumption \ref{assum_step}.
			In the following, we present a uniform bound on the iterates $\{\boldsymbol{x}^{(t)}\}_{t\geq0}$ generated by Algorithm \ref{alg:CDSA}.	The   proof  is presented in Appendix \ref{appe_hatX}.
			
			\begin{lemma}\label{lemma_hatX}
				Let Assumptions \ref{assum_graph} and \ref{assum_func} hold. Considering Algorithm \ref{alg:CDSA} with stepsize $\alpha^{(t)}$ of order $\mathcal{O}(\frac{1}{t})$, for all $t\geq 0$ we have the gap between the iteration vector $\boldsymbol{x}^{(t)}$ which defined in \eqref{mathbbx} and the optimal solution under belief  $\tilde{\boldsymbol{q}}$ which defined in \eqref{converq} is bounded by some constant $\hat{X}$, i.e.
				\begin{align}
					\|\boldsymbol{x}^{(t)}-\boldsymbol{1}x^*(\tilde{\boldsymbol{q}})\|^2\leq \hat{X}.
				\end{align}
			\end{lemma}

			Next, we derive the convergence rate of consensus error based on the recursive  form of Lemma \ref{lemma_consens},
 while present it in a more general way. For completeness, its proof is given in Appendix \ref{appe_rec_e}.
			\begin{lemma}\label{lemma_rec_e}
				Let $\{e^{(t)}\}_{t\geq 0}$ and $\{\alpha^{(t)}\}_{t\geq0}$ be nonnegative sequences, where  $\alpha^{(t)}$ of order $\mathcal{O}(\frac{1}{t})$. If the   recursion
				\begin{equation}\label{rec-et}
					e^{(t+1)}\leq \delta e^{(t)}+ c[\alpha^{(t)}]^2
				\end{equation}
				holds  for $\delta\in(0,1)$ and $c>0$. Then the sequence $\{e^{(t)}\}_{t\geq0}$ diminishes to $0$ with rate $\mathcal{O}(\frac{1}{t^2})$.
			\end{lemma}

			In addition, we  introduce the following lemma from  \cite[lemma 1]{ahmadi2020resolution} for converge analysis. 
			\begin{lemma}\label{lemma_con}
				Let the sequence recursion
				\begin{equation}
					u^{(t+1)}\leq p^{(t)}u^{(t)}+\beta^{(t)}
				\end{equation}
				hold for $0 \leq p^{(t)}< 1, \beta^{(t)}\geq 0, \sum_{t=1}^{\infty}(1-p^{(t)})=\infty$ and $\lim_{t\rightarrow \infty}\frac{\beta^{(t)}}{(1-p^{(t)})}=0$.  If $u^{(t)}\geq 0$,   we have $\lim_{t\rightarrow \infty}u^{(t)}=0$.
			\end{lemma}
			\begin{theorem}\label{thm2}
				Let Assumptions \ref{assum_boundbelief}, \ref{assum_graph}, \ref{assum_step}, and \ref{assum_func} hold. Consider Algorithm \ref{alg:CDSA} with the stepsize $\alpha^{(t)}$ of order $\mathcal{O}(\frac{1}{t})$. Then for every agent $i\in\mathcal{N}$, the decision  sequence $x_i^{(t)}$   converges to an optimal solution of \eqref{newfunc} under $\tilde{\boldsymbol{q}}$, i.e. $\lim_{t\rightarrow\infty}x_i^{(t)}=x^*(\tilde{\boldsymbol{q}})$.
			\end{theorem}
			\begin{proof}
				By Lemma \ref{lemma_consens} and Lemma  \ref{lemma_hatX}, we define
				 $e^{(t)}=\|\boldsymbol{x}^{(t)}-\boldsymbol{1}\bar{x}^{(t)}\|^2 $, $\delta=\frac{3+\rho_w^2}{4}$,  and
					\begin{align*}
						&c=\frac{3 \rho_w^2}{1-\rho_w^2} \left[2 M^2 L^2\hat{X}+ 2 M  \sum_{i=1}^N\|\nabla_x \boldsymbol{J}_i(x^*(\tilde{\boldsymbol{q}}), \boldsymbol{\theta})\|^2 \right].
					\end{align*}
		Then we can recast  Lemma   \ref{lemma_consens} as  the recursion of Lemma \ref{lemma_rec_e}.
						Since $\rho_w\in [0,1), $ we have $\delta\in [3/4,1)$. Then by using Lemma \ref{lemma_rec_e}, we conclude that    the consensus error $\|\boldsymbol{x}^{(t)}-\boldsymbol{1}\bar{x}^{(t)}\|^2$ diminishes to $0$ at rate $\mathcal{O}(\frac{1}{t^2})$.
				
				Besides, in light of Lemma \ref{lemma_con} and Lemma  \ref{lemma_op}, we set
				\begin{align}
					u^{(t)}:=&\|\bar{x}^{(t)}-x^*(\tilde{\boldsymbol{q}})\|\notag,\\
					p^{(t)}:=&\sqrt{1-\alpha^{(t)}\mu+2\mu L[\alpha^{(t)}]^2}\notag,\\
					\beta^{(t)}:=&\frac{ [\alpha^{(t)}]^{0.5}L}{\sqrt{\mu N}} \|\boldsymbol{x}^{(t)}-\boldsymbol{1}\bar{x}^{(t)}\|+\frac{\sqrt{2}L\alpha^{(t)}}{\sqrt{N}} \|\boldsymbol{x}^{(t)}-\boldsymbol{1}\bar{x}^{(t)}\|\notag\\
					&~+\alpha^{(t)}\frac{1}{N}\sum_{i=1}^{N}\|\boldsymbol{q}_i^{(t)}-\tilde{\boldsymbol{q}} \| \| \nabla_x\boldsymbol{J}_i(x^*(\tilde{\boldsymbol{q}}),\boldsymbol{\theta})\|. \notag
				\end{align}
			Since $\alpha^{(t)}<\frac{1}{2L}$, $0\leq1-\alpha^{(t)}\mu(1-2\alpha^{(t)}L)<1$, therefore $0\leq p^{(t)}<1$. Obviously, $\beta^{(t)}\geq0$.
				Note that
				\begin{align}
					\lim_{y\rightarrow0}\frac{1-\sqrt{1-y}}{0.5y}
					\xlongequal[y=2z-z^2]{z=1-\sqrt{1-y}}\lim_{z\rightarrow0}\frac{z}{z-0.5z^2}=1.
				\end{align}
			Thus, getting limit with substitution of equivalence infinitesimal, we have $\left(1-p^{(t)}\right)\sim \left(0.5\alpha^{(t)}\mu-\mu L[\alpha^{(t)}]^2\right)$.
			Therefore,  by recalling $\sum\limits_{t=1}^{\infty}\alpha^{(t)}=\infty$ from Assumption \ref{assum_step}, we have
			\begin{align}
				\sum_{t=1}^{\infty}(1-p^{(t)})=\sum_{t=1}^{\infty}\left(0.5\alpha^{(t)}\mu-\mu L[\alpha^{(t)}]^2\right)\!=\infty.\label{condition1}
				\end{align}

			Consider
			\begin{align}
				&\lim_{t\rightarrow\infty}\frac{\beta^{(t)}}{1-p^{(t)}}=\frac{L}{\sqrt{\mu N}}\lim_{t\rightarrow\infty}\frac{[\alpha^{(t)}]^{0.5}\|\boldsymbol{x}^{(t)}-\boldsymbol{1}\bar{x}^{(t)}\|}{0.5\alpha^{(t)}\mu-\mu L[\alpha^{(t)}]^2}\notag\\
				&\quad+\sqrt{2}L\lim_{t\rightarrow\infty}\frac{\alpha^{(t)}L\|\boldsymbol{x}^{(t)}-\boldsymbol{1}\bar{x}^{(t)}\|}{0.5\alpha^{(t)}\mu-\mu L[\alpha^{(t)}]^2}\notag\\
				&~+\lim_{t\rightarrow\infty}\frac{\alpha^{(t)}\frac{1}{N}\sum_{i=1}^{N}\|\boldsymbol{q}_i^{(t)}-\tilde{\boldsymbol{q}} \| \| \nabla_x\boldsymbol{J}_i(x^*(\tilde{\boldsymbol{q}}),\boldsymbol{\theta})\|}{0.5\alpha^{(t)}\mu-\mu L[\alpha^{(t)}]^2}.\label{con2}
			\end{align}
		Since $\alpha^{(t)}=\mathcal{O}(\frac{1}{t})$ and $\|\boldsymbol{x}^{(t)}-\boldsymbol{1}\bar{x}^{(t)}\|=\mathcal{O}(\frac{1}{t})$ when $t\rightarrow \infty$, we can conclude that the limit of the first two terms of \eqref{con2} is $0$.
		As for the last term of \eqref{con2}, recalling Theorem \ref{lemma_be_con}, we have $\lim_{t\rightarrow\infty}\|\boldsymbol{q}_i^{(t)}-\tilde{\boldsymbol{q}} \| =0$. Together with $\| \nabla_x\boldsymbol{J}_i(x^*(\tilde{\boldsymbol{q}}),\boldsymbol{\theta})\|$ is bounded with a fixed point, we can obtain that the limit of the last term of \eqref{con2} also comes to $0$.
		As a result,
		\begin{align}
			\lim_{t\rightarrow\infty}\frac{\beta^{(t)}}{1-p^{(t)}}=0\label{condition2}.
		\end{align}
	
	Combining $0\leq p^{(t)}<1$ and  $\beta^{(t)}\geq 0$, together with \eqref{condition1} and \eqref{condition2}, we see that the conditions of Lemma \ref{lemma_con} hold. Therefore, by applying  Lemma \ref{lemma_con}, we conclude that $u^{(t)}\rightarrow 0$ as $t\rightarrow0$, i.e. $\|\bar{x}^{(t)}-x^*(\tilde{\boldsymbol{q}})\|\rightarrow0$.

Therefore, by recalling that $\|\boldsymbol{x}^{(t)}-\boldsymbol{1}\bar{x}\|^2\rightarrow0$ and $\|\bar{x}^{(t)}-x^*(\tilde{\boldsymbol{q}})\|^2\rightarrow0$ with $t\rightarrow\infty$, we achieve
				 \begin{align} \|\boldsymbol{x}^{(t)}-\boldsymbol{1}x^*(\tilde{\boldsymbol{q}})\|^2=\|\boldsymbol{x}^{(t)}-\boldsymbol{1}\bar{x}^{(t)}+\boldsymbol{1}\bar{x}^{(t)}-\boldsymbol{1}x^*(\tilde{\boldsymbol{q}})\|^2\notag\\					  \leq2\|\boldsymbol{x}^{(t)}-\boldsymbol{1}\bar{x}\|^2+2\|\boldsymbol{1}\bar{x}-\boldsymbol{1}x^*(\tilde{\boldsymbol{q}})\|^2\notag\\
 =2\|\boldsymbol{x}^{(t)}-\boldsymbol{1}\bar{x}\|^2+2N\|\bar{x}^{(t)}-x^*(\tilde{\boldsymbol{q}})\|^2 \rightarrow 0\notag,
					 \end{align}
  Hence for all $i\in\mathcal{N}$, $\lim_{t\rightarrow\infty}x_i^{(t)}=x^*(\tilde{\boldsymbol{q}})$.
			\end{proof}
			%

			\subsection{Convergence to the True Solution}
			
			Though the algorithm can converge to $x^*(\tilde{\boldsymbol{q}})$ based on Theorem 1 and Theorem 2, whether it can converge to the true solution $x^*(\boldsymbol{q}^*)$ remains unknown.  In the following,  we will validate that $\tilde{\boldsymbol{q}}=\boldsymbol{q}^*.$
			First of all, we introduce Toeplitz's lemma \cite{knopp1990theory} to help develop the convergence result.
		
			\begin{lemma}
				Let $\{A_{nk}, 1\leq k\leq k_n\}_{n\geq1}$ be a double array of positive numbers such that for fixed $k$, $A_{nk}\rightarrow 0$ when $n\rightarrow\infty$. Let $\{Y_n\}_{n\geq 1}$ be a sequence of real numbers. If $Y_n\rightarrow y$ and $\sum_{k=1}^{k_n}A_{nk}\rightarrow 1$ when $n\rightarrow \infty$, then $\lim\limits_{n \rightarrow \infty} \sum_{k=1}^{k_n} A_{nk}Y_k=y$.\label{Toep}
			\end{lemma}
			
			Based on which, we obtain the following Theorem.
			\begin{theorem}\label{thm3}
				Let Assumptions \ref{assum_boundbelief}, \ref{assum_graph}, \ref{assum_step}, \ref{assum_func},  and \ref{assum_optimal} hold  with the stepsize $\alpha^{(t)}$ of order $\mathcal{O}(\frac{1}{t})$. Consider the belief sequence $\{q_i^{(t)}\}_{t\geq 0}$ generated by Algorithm \ref{alg:CDSA}. Then, every agent's estimate almost surely converges to the true parameter $\theta_*$. In addition , for each agents $i\in\mathcal{N}$ and $\theta_m\neq\theta_*$,
					\begin{align}
					q_i^{(T+1)}\left(\theta_m\right) \leq \exp \left(-Z(\theta_*,\theta_m)\sum_{t=1}^T \alpha^{(t)} \right) \quad \text { a.s. }\label{belief_rate}
				\end{align}
			where
			\begin{small}
					\begin{align*}
					Z(\theta_*,\theta_m)\!=\!\frac{1}{N}\!\sum_{j=1}^N \! D_{K L}\!\left(f_j\left(y_j | x^*(\tilde{\boldsymbol{q}}), \theta_*\right)\! \| f_j\left(y_j | x^*(\tilde{\boldsymbol{q}}), \theta_m\right)\right).
				\end{align*}
			\end{small}
			\end{theorem}
			\begin{proof}
				 Firstly, we give an equivalent form of $\lim _{T \rightarrow \infty}\frac{1}{\sum_{t=1}^T \alpha^{(t)}} \log \frac{q_i^{(T+1)}\left(\theta_*\right)}{q_i^{(T+1)}\left(\theta_m\right)}$. This is the preparation for the later use of Lemma \ref{Toep} to derive the overall convergence.
				 Based on the belief update rules (\ref{b_up}) and (\ref{q_up}),  and similarly to the derivation of   \eqref{log_qq}, we derive
				\begin{align}
					&   \log \frac{q_i^{(T+1)}\left(\theta_*\right)}{q_i^{(T+1)}\left(\theta_m\right)} \notag
					\\  =&  \sum_{j=1}^N \sum_{t=1}^T W^t(i,j) \alpha^{(T-t+1)} z_j^{(T-t+1)}(\theta_*,\theta_m) , \label{bd-logq}
				\end{align}
				where $z_j^{(t)}(\theta_*,\theta_m)=\log \frac{f_j\left(y_j^{(t)} \mid x_j^{(t)}, \theta_*\right)}{f_j\left(y_j^{(t)} \mid x_j^{(t)}, \theta_m\right)}$.
				With Assumption \ref{assum_graph}, we achieve
				the double stochasticity of  $W^t$.  Then by using  \eqref{bd-logq},  we have
				\begin{align}
					&\frac{1}{N}\sum_{i=1}^{N}\log \frac{q_i^{(T+1)}\left(\theta_*\right)}{q_i^{(T+1)}\left(\theta_m\right)}\notag\\
					&=\frac{1}{N}\sum_{i=1}^{N}\sum_{j=1}^N \sum_{t=1}^T W^t(i,j) \alpha^{(T-t+1)} z_j^{(T-t+1)}(\theta_*,\theta_m)\notag\\
					&=\frac{1}{N}\sum_{j=1}^{N}\sum_{t=1}^T \alpha^{(T-t+1)}  z_j^{(T-t+1)}(\theta_*,\theta_m) \notag
					\\&= \frac{1}{N}\sum_{j=1}^{N}\sum_{t=1}^T \alpha^{(t)}  z_j^{(t)}(\theta_*,\theta_m).\label{equ_logq}
				\end{align}
				
			%
			
			By  utilizing  Lemma \ref{lemma_logcon} and Assumption \ref{assum_step}, 
			\begin{align}
				& \lim _{T \rightarrow \infty} \tfrac{1}{\sum_{t=1}^T \alpha^{(t)}}\left(\log \tfrac{q_i^{(T+1)}\left(\theta_*\right)}{q_i^{(T+1)}\left(\theta_m\right)}-\tfrac{1}{N} \sum_{i=1}^N \log \tfrac{q_i^{(T+1)}\left(\theta_*\right)}{q_i^{(T+1)}\left(\theta_m\right)}\right) \notag\\
				&	= 0.
			\end{align}
			Therefore,
			\begin{align}
				&	\lim _{T \rightarrow \infty}\frac{1}{\sum_{t=1}^T \alpha^{(t)}} \log \frac{q_i^{(T+1)}\left(\theta_*\right)}{q_i^{(T+1)}\left(\theta_m\right)}\notag\\
				&= \lim _{T \rightarrow \infty} \tfrac{1}{\sum_{t=1}^T \alpha^{(t)}} \tfrac{1}{N} \sum_{i=1}^N \log \tfrac{q_i^{(T+1)}\left(\theta_*\right)}{q_i^{(T+1)}\left(\theta_m\right)}  \notag \\
				& \overset{\eqref{equ_logq}}{=} 	 \frac{1}{N} \sum_{j=1}^N  \lim _{T \rightarrow \infty} \frac{1}{\sum_{t=1}^T \alpha^{(t)}}\sum_{t=1}^T \alpha^{(t)} z_j^{(t)}\left(\theta_*, \theta_m\right).\label{sim_zq}
			\end{align}
			
			To investigate  the convergence of the above equation, we first study the convergence of $\frac{1}{T} \sum_{t=1}^T z_j^{(t)}\left(\theta_*, \theta_m\right)$ by newly shaped random variable. Based on which we can later use strong Large Number Theorem in the limit case.
			
			Denote the cumulative distribution function as follow
			\begin{align*}
				G_j^{(t)}(z)\triangleq Pr\left(\log \frac{f_j\left(y_j \mid x_j^{(t)}, \theta_*\right)}{f_j\left(y_j \mid x_j^{(t)}, \theta_m\right)}\leq z\right),\\
				G_j^*(z)\triangleq Pr\left(\log \frac{f_j\left(y_j \mid x^*(\tilde{\boldsymbol{q}}), \theta_*\right)}{f_j\left(y_j \mid x^*(\tilde{\boldsymbol{q}}), \theta_m\right)}\leq z\right).
			\end{align*}
			Then, since $x_j^{(t)}\rightarrow x^*(\tilde{\boldsymbol{q}})$ as $t\rightarrow\infty$ and by the continuity of the likelihood function (Assumption \ref{assum_boundbelief}), we have
			\begin{align}
				\lim_{t \rightarrow \infty} G_j^{(t)}(z)=G_j^*(z), ~\forall z\in \mathbb{R}.\label{lim_G}
			\end{align}
			For any sequence of realized outcomes $\{(x_j^{(t)},y_j^{(t)})\}_{t=1}^{\infty}$, we define a sequence of random variable $\{\Delta_j^{(t)}\}_{t=1}^{\infty}$, where $\Delta_j^{(t)}\triangleq G_j^{(t)}(z_j^{(t)}(\theta_*,\theta_m))$. Then $\Delta_j^{(t)}\in [0,1]$, and for any $\beta\in [0,1]$,
			\begin{align*}
				Pr(\Delta_j^{(t)}\leq \beta)&=Pr\left(G_j^{(t)}(z_j^{(t)}(\theta_*,\theta_m))\leq\beta\right)\\
				&=Pr\left(z_j^{(t)}(\theta_*,\theta_m)\leq (G_j^{(t)})^{-1}(\beta)\right)\\
				&=G_j^{(t)}(G_j^{(t)})^{-1}(\beta)=\beta.
			\end{align*}
			That is, $\Delta_j^{(t)}$ is independent and uniformly distributed on $[0,1]$.
			
			Consider another sequence of random variables $\{\eta_j^{(t)}\}_{t=1}^{\infty}$, where $\eta_j^{(t)}\triangleq(G_j^*)^{-1}(\Delta_j^{(t)})$. Since $\Delta_j^{(t)}$ is i.i.d with uniform distribution, $\eta_j^{(t)}$ is also i.i.d with the same distribution as $\log \frac{f_j\left(y_j \mid x^*(\tilde{\boldsymbol{q}}), \theta_*\right)}{f_j\left(y_j \mid x^*(\tilde{\boldsymbol{q}}), \theta_m\right)}$. Additionally, since each $\Delta_j^{(t)}$ is generated from the realized outcome $(x_j^{(t)},y_j^{(t)})$, $(\eta_j^{(t)})_{t=1}^{\infty}$ is in the same probability space as $z_j^{(t)}(\theta_*,\theta_m)$. From \eqref{lim_G}, $G_j^{(t)}$ converge to $G_j^*$ as $t\rightarrow \infty$. Therefore, with probability $1$,
			\begin{align}
				&\lim_{t\rightarrow \infty} \left|z_j^{(t)}(\theta_*,\theta_m)-\eta_j^{(t)}\right|\notag\\
				=&\lim_{t \rightarrow \infty}\left|z_j^{(t)}(\theta_*,\theta_m)-(G_j^*)^{-1}\left(G_j^{(t)}\left(z_j^{(t)}(\theta_*,\theta_m)\right)\right)\right|\notag\\
				=&0\notag
			\end{align}
			Consequently, w.p.1
			\begin{align}
				\lim_{T\rightarrow\infty}\left|\frac{1}{T}\sum_{t=1}^{T}\left(z_j^{(t)}(\theta_*,\theta_m)-\eta_j^{(t)}\right)\right|\notag\\
				\leq \lim_{T \rightarrow \infty}\frac{1}{T}\sum_{t=1}^{T}\left|z_j^{(t)}(\theta_*,\theta_m)-\eta_j^{(t)}\right|=0.
			\end{align}
This  together with $(\eta_j^{(t)})_{t=1}^{\infty}$ is i.i.d with  the distribution of
			$\log \frac{f_j\left(y_j \mid x^*(\tilde{\boldsymbol{q}}), \theta_*\right)}{f_j\left(y_j \mid x^*(\tilde{\boldsymbol{q}}), \theta_m\right)}$, by the strong Large Number Theorem
			\begin{align} \label{as_sez}
				\lim _{T \rightarrow \infty} \frac{1}{T} \sum_{t=1}^T z_j^{(t)}\left(\theta_*, \theta_m\right)=\frac{1}{T}\sum_{t=1}^{T}\eta_j^{(t)}\notag\\
				=\mathbb{E}\left[\log \frac{f_j\left(y_j \mid x^*(\tilde{\boldsymbol{q}}), \theta_*\right)}{f_j\left(y_j \mid x^*(\tilde{\boldsymbol{q}}), \theta_m\right)}\right]\quad a.s.
			\end{align}
			
			Secondly, recalling the equivalence form in (\ref{sim_zq}), we use double array convergence principle in Lemma \ref{Toep} to derive the convergence of $\lim _{T \rightarrow \infty}\frac{1}{\sum_{t=1}^T \alpha^{(t)}} \log \frac{q_i^{(T+1)}\left(\theta_*\right)}{q_i^{(T+1)}\left(\theta_m\right)}$. One of the array shows in above result (\ref{as_sez}), the other is created by mathematical technique as follow. 
			
			Note that $T \alpha^{(T)}+\sum_{t=1}^{T-1} t\left(\alpha^{(t)}-\alpha^{(t+1)}\right)=\sum_{t=1}^T \alpha^{(t)}$.  Define $Y_t=\frac{1}{t} \sum_{\tau=1}^t z_j^{(\tau)}\left(\theta_*, \theta_m\right)$, and the sequence
			$\{A_{T t}, 1\leq t\leq T\}_{T\geq 1}$ with $A_{T t}=\frac{t\left(\alpha^{(t)}-\alpha^{(t+1)}\right)}{\sum_{t=1}^T \alpha^{(t)}}(t=1, \ldots, T-1), A_{T T}=\frac{T \alpha^{(T)}}{\sum_{t=1}^T \alpha^{(t)}}$. Then from (\ref{sim_zq}) we derive
			\begin{align*}
				& \lim _{T \rightarrow \infty} \frac{1}{\sum_{t=1}^T \alpha^{(t)}} \log \frac{q_i^{(T+1)}\left(\theta_*\right)}{q_i^{(T+1)}\left(\theta_m\right)} \notag\\
				= & \frac{1}{N} \sum_{j=1}^N \lim _{T \rightarrow \infty} \frac{1}{\sum_{t=1}^T \alpha^{(t)}}\left(T \alpha^{(T)} \cdot \frac{1}{T} \sum_{t=1}^T z_j^{(t)}\left(\theta_*, \theta_k\right)\right.\notag\\
				&\left.+\sum_{t=1}^{T-1} t\left(\alpha^{(t)}-\alpha^{(t+1)}\right) \cdot \frac{1}{t} \sum_{\tau=1}^t z_j^{(\tau)}\left(\theta_*, \theta_m\right)\right) \notag
				\\& = \frac{1}{N} \sum_{j=1}^N \lim _{T \rightarrow \infty}  \sum_{t=1}^TA_{T t} Y_t .
			\end{align*}
			By noticing that  $\sum_{t=1}^T A_{T t}=1,$ and the almost sure convergence of $\{Y_t\}$ from  \eqref{as_sez}, we  conclude  from Lemma \ref{Toep} that the following holds almost surely.
			\begin{align}\label{limit-logq}
				&\lim _{T \rightarrow \infty} \frac{1}{\sum_{t=1}^T \alpha^{(t)}} \log \frac{q_i^{(T+1)}\left(\theta_*\right)}{q_i^{(T+1)}\left(\theta_m\right)} \notag\\
				& =\frac{1}{N} \sum_{j=1}^N \lim _{T \rightarrow \infty} \frac{1}{T} \sum_{t=1}^T z_j^{(t)}\left(\theta_*, \theta_m\right) \notag\\
				&\overset{ \eqref{as_sez}}{ =}\mathbb{E}\left[\frac{1}{N} \sum_{j=1}^N \log \frac{f_j\left(y_j \mid x^*(\tilde{\boldsymbol{q}}), \theta_*\right)}{f_j\left(y_j \mid x^*(\tilde{\boldsymbol{q}}), \theta_m\right)}\right]\\
				&=\frac{1}{N} \sum_{j=1}^N D_{K L}\left(f_j\left(y_j \mid x^*(\tilde{\boldsymbol{q}}), \theta_*\right) \| f_j\left(y_j \mid x^*(\tilde{\boldsymbol{q}}), \theta_m\right)\right). \notag
			\end{align}
		
		Finally, we can derive the belief convergence rate based on the properties of beliefs and above result. 
			By recalling Assumption \ref{assum_optimal}, we obtain $Z(\theta_*,\theta_m)\triangleq\frac{1}{N} \sum_{j=1}^N D_{K L}\left(f_j\left(y_j \mid x^*(\tilde{\boldsymbol{q}}), \theta_*\right) \| f_j\left(y_j \mid x^*(\tilde{\boldsymbol{q}}), \theta_m\right)\right)>0$.
			Therefore, \eqref{limit-logq}  indicates that for all $\epsilon>0$, there exists $T'(\epsilon)$ such that for all $T>T'$,
			\begin{align}
				\left|\frac{1}{\sum_{t=1}^T \alpha^{(t)}} \log  \frac{q_i^{(T+1)}\left(\theta_*\right)}{q_i^{(T+1)}\left(\theta_m\right)}-Z_j(\theta_*,\theta_m)\right| \leq \epsilon \quad \text { a.s. }\notag
			\end{align}
			As a result,
			\begin{align}
				\frac{q_i^{(T+1)}\left(\theta_m\right)}{q_i^{(T+1)}\left(\theta_*\right)} \leq \exp \left(-\sum_{t=1}^T \alpha^{(t)}\big( Z(\theta_*,\theta_m)-\epsilon\big)\right) \quad \text { a.s. }\label{rateqo}
			\end{align}
		Using the fact that $\sum_{m=1}^{M}q_i^{T+1}(\theta_m)=1$, we obtain
		\begin{align}
			\tfrac{1}{q_i^{(T+1)}\left(\theta_*\right)} -1\leq \!\!\sum_{\theta_m\neq\theta_*}\!\!\exp \left(-\sum_{t=1}^T \alpha^{(t)}\left(Z(\theta_*,\theta_m)-\epsilon\right)\right) ,~ \text { a.s. }\notag
		\end{align}
		Furthermore, we derive
		\begin{align}
			\tfrac{1}{1+\sum_{\theta_m\neq\theta_*}\exp \left(-\sum_{t=1}^T \alpha^{(t)}\big(Z(\theta_*,\theta_m)-\epsilon\big)\right)}\leq q_i^{(T+1)}\left(\theta_*\right)\leq 1,~ \text{a.s.}
		\end{align}
		Because of $\sum_{t=1}^T \alpha^{(t)}\rightarrow \infty$, then $q_i^{t}(\theta_*)\rightarrow 1$ a.s.
		We then conclude from Theorem \ref{lemma_be_con} that  $\tilde{\boldsymbol{q}}=\boldsymbol{q}^*,$
		where
		in vector $\boldsymbol{q}^*(\boldsymbol{\theta})$ only $q(\theta_*)=1$, while  other $q(\theta_m)|_{\theta_m\neq \theta_*}=0$.
		
		Besides, since in \eqref{rateqo} $\epsilon$ is arbitrary and $q_i^{(T+1)}\left(\theta_*\right)\leq1$, we can obtain that for any $i\in\mathcal{N}$, $\theta_m\neq\theta_*$,
		\begin{align}
		q_i^{(T+1)}\left(\theta_m\right) \leq \exp \left(-Z(\theta_*,\theta_m)\sum_{t=1}^T \alpha^{(t)} \right) \quad \text { a.s. }\notag
		\end{align}
	This completes the assertion of the theorem.
		\end{proof}
		
		\textbf{Remark 4. } Since the stepsize $\alpha^{(t)}$  is of order  $O(\frac{1}{t})$, we conclude $\sum_{t=1}^T \alpha^{(t)} =O(\ln (t))$. As a result, based on Theorem \ref{thm3}, we can obtain that for each agent $i\in\mathcal{N}$ and  $\theta_m\neq\theta_*$, the belief sequence can reach a sublinear convergence rate, i.e. $q_i^{(T+1)}\left(\theta_m\right)=O(1/T)$.

	Overall, recalling that $x^*(\boldsymbol{q}^*)=x_*$, Theorem \ref{thm2} together with Theorem \ref{thm3} implies that the algorithm converges to its true optimal solution $x_*$. We formalize  it
	in the following result.
			\begin{theorem}
	Let Assumptions \ref{assum_boundbelief}-\ref{assum_optimal} hold. Consider Algorithm \ref{alg:CDSA} hold with  the stepsize $\alpha^{(t)}$ of order $\mathcal{O}(\frac{1}{t})$. Then for every agent $i\in\mathcal{N}$,
	\[\lim_{t\rightarrow\infty}x_i^{(t)}=x_* , \quad a.s. \]
\end{theorem}
		\section{Experiments}
		In this section, we provide  numerical examples to demonstrate our theoretical analysis. One is the near-sharp quadratic problem,
and the other considers the  scenario    of source  searching.
		\subsection{ Near-sharp Quadratic Problem}
		Consider the following near-sharp quadratic problem:
		\begin{equation}
			\min_{\boldsymbol{x}\in \mathbb{R}^p} \frac{1}{N}\sum_{i=1}^{N}\|\theta_*x-d_i\|^2,\label{QP}
		\end{equation}
		where $d_i=e_i\boldsymbol{1}$ and $e_i$ is the $i$-th smallest eigenvalue of the $W$. Set $\Theta=\{1,2.5,4\}$ and $\theta_*=2.5$. For all agent $i$, the realized date is obtained from (\ref{y_data}), where $\epsilon_i\sim N(0,1)$.
		
		Considering five agents communicate under path topology, we use Algorithm \ref{alg:CDSA} to solve the problem (\ref{QP})  with  the stepsize chosen as $\alpha^{(t)}=\frac{10}{t+80}$. Set the weighted adjacency  matrix by Metropolis-Hastings rules\cite{fastdis_ave}. We show the average beliefs $\bar{q}^{(t)}=\frac{1}{5}\sum_{i=1}^{5}q_i^{(t)}$  of five agents for the three possible parameters in Figure \ref{figure1}, and the gap between each agent's belief $q_i^{(t)}(\theta_*)$ and average belief $\bar{q}^{(t)}(\theta_*)$, i.e. $q_i^{(t)}(\theta_*)-\bar{q}^{(t)}(\theta_*)$ for all $i\in \mathcal{N}$ in Figure \ref{figure2}.
From Figure \ref{figure1}, we can see that the posterior probability of true parameter converge to $1$ and the probability of fake parameter decrease to $0$, which means the average belief sequence generated by our Algorithm converges to the true parameter.  Figure \ref{figure2} shows that the gap between each agent's belief of the true parameter and the average belief is $0$ at the very beginning, which is because we set $q_i^{(0)}=\frac{1}{M}\boldsymbol{1}_M$ for all $i\in\mathcal{N}$ in the algorithm initialization.
As the iteration of the algorithm proceeds, initially each agent has not yet fully communicated with its neighbors to integrate global information, and thus  cannot reach   consensus. Gradually, all agents beliefs get consensus to the true parameter.
%
%
	\begin{figure}
			\centering
		\includegraphics[width=0.8\columnwidth]{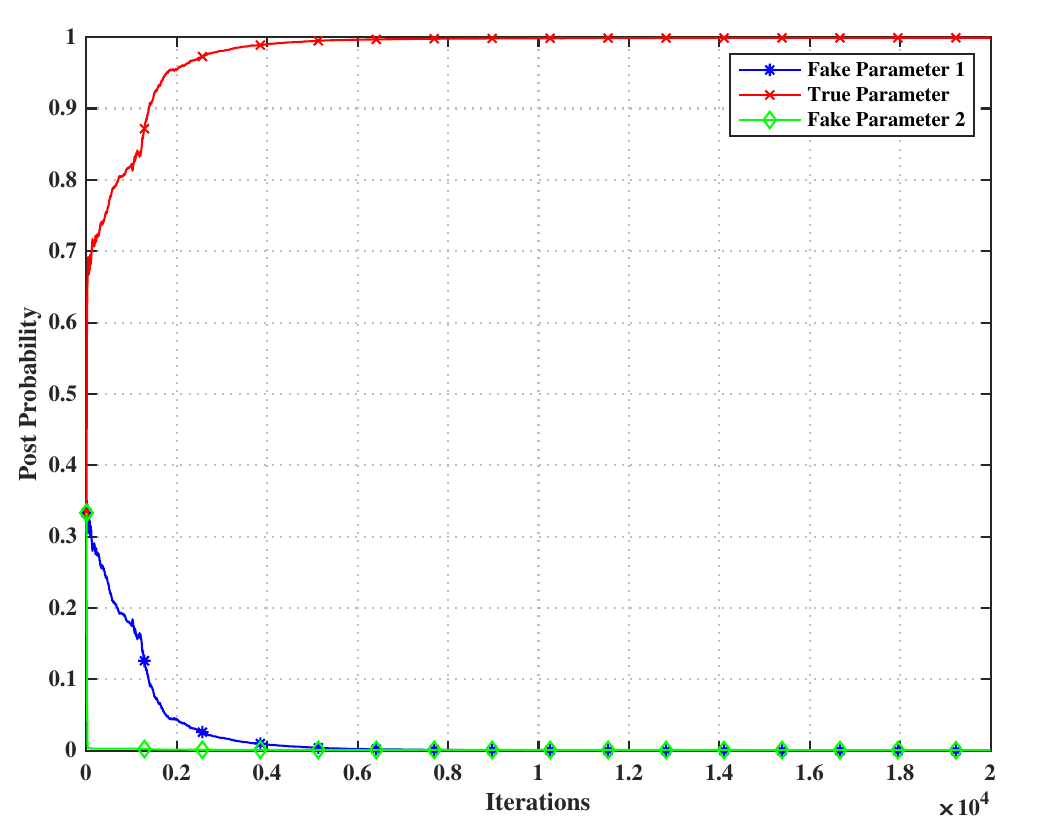}
		\caption{The average belief  of five agents for three candidate  parameters }
		\label{figure1}
	\end{figure}

\begin{figure}
		\centering
	\includegraphics[width=0.8\columnwidth]{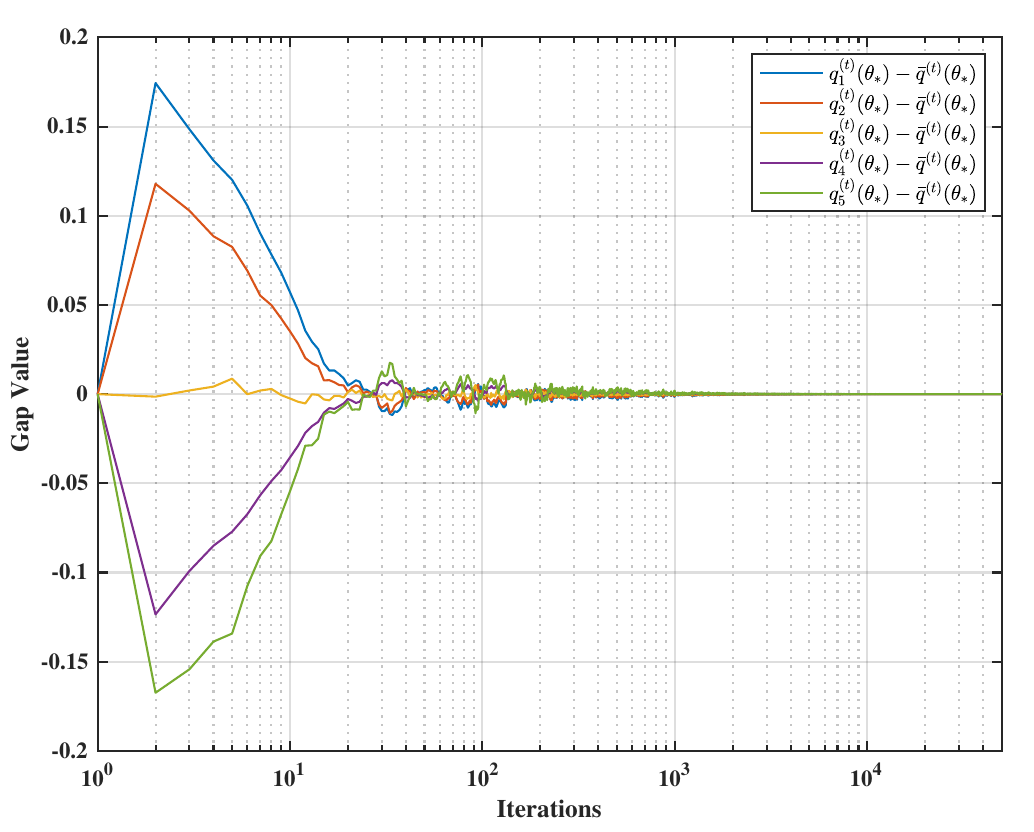}
	\caption{The gap between average belief and each agent's belief of true parameter}
	\label{figure2}
\end{figure}

Furthermore, the adaptive decision sequences of all agents are presented in Figure \ref{figure3}. We can see that five agents' decision reach  consensus to  the   true optimal decision.
		
		\begin{figure}[h]
			\centerline{\includegraphics[width=0.8\columnwidth]{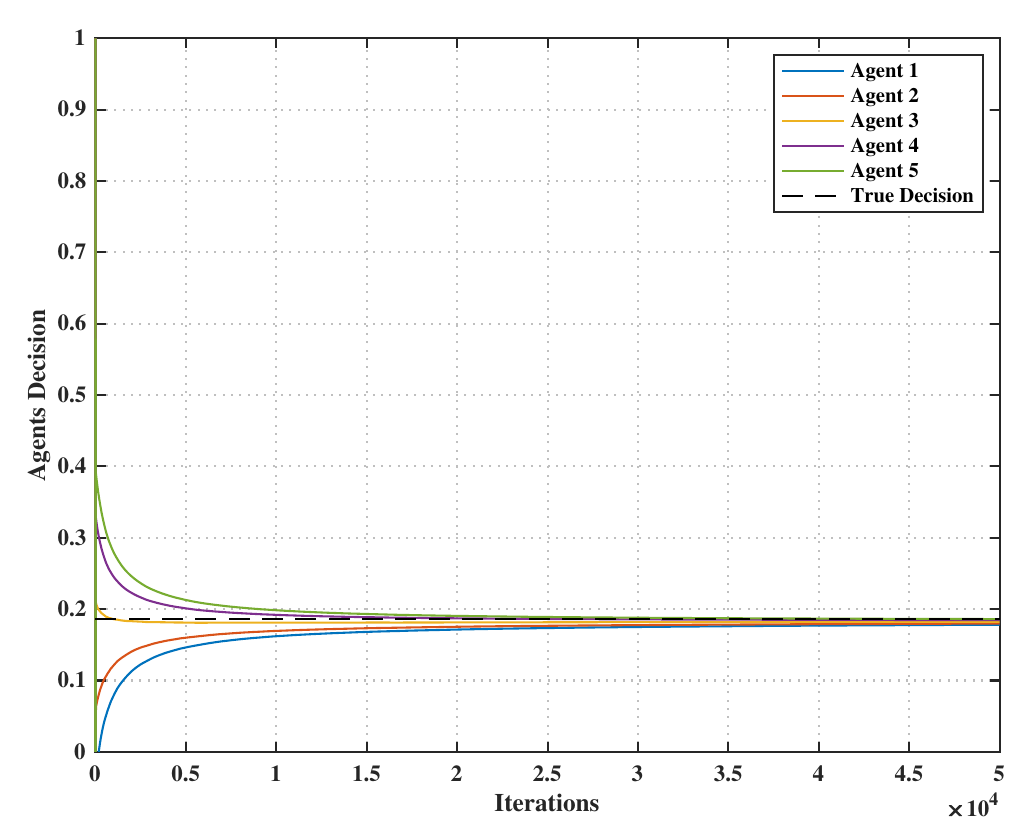}}
			\caption{Decision convergence of all agents under Algorithm \ref{alg:CDSA} }
			\label{figure3}
		\end{figure}
		
		{\bf The impact of stepsizes.}  Besides, we  implement Algorithm \ref{alg:CDSA}  with   different stepsizes to explore their impact on the algorithm convergence. The   beliefs of the true parameter    with stepsizes $\alpha^{(t)}=\frac{1}{t+3}, \frac{1}{t+5}, \frac{10}{t+80}$ are shown  in Figure \ref{alpha_com}. Since $\frac{10}{t+80}>\frac{1}{t+3}>\frac{1}{t+5}$ for any $t\geq6$, and $\sum_{t=1}^{T}\frac{10}{t+80}>\sum_{t=1}^{T}\frac{1}{t+3}>\sum_{t=1}^{T}\frac{1}{t+5}$ for any $T\geq 15$.
Based on  the convergence rate \eqref{belief_rate} of beliefs, we can obtain that as $T\rightarrow\infty$, algorithm implement with stepsize $\frac{10}{t+80}$  converge faster than others. Whereas at the beginning when $T$ is small, due to $\sum_{t=1}^{T}\frac{1}{t+3}>\sum_{t=1}^{T}\frac{10}{t+80}$, algorithm with stepsize $\frac{1}{t+3}$ performs better. The theoretical results match the numerical results in Figure \ref{alpha_com}. Generally speaking, algorithm with bigger stepsize leads to faster convergence  rate  as data information used is much more efficient than prior information due to Equation \eqref{b_up}.

\begin{figure}
	\centering
	\includegraphics[width=0.8\columnwidth]{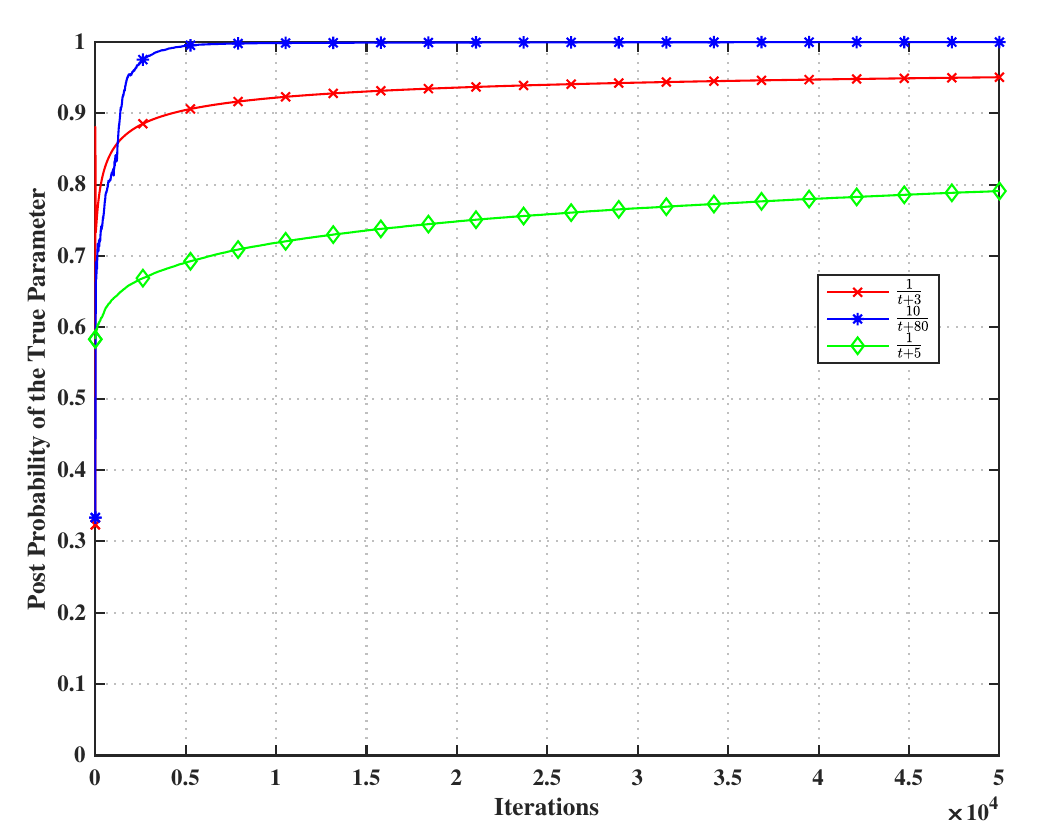}
	\caption{The impact of   stepsizes    on the convergence rate of the true parameter's belief}
	\label{alpha_com}
\end{figure}
%
%

		
		{\bf Different distributed consensus protocol comparison.}
		We further carry out simulations to compare the  classical distributed linear consensus protocol \cite{fastdis_ave} with \eqref{q_up} which implements distributed consensus averaging on a reweighting of the log-belief.  The result  demonstrated in Figure \ref{log_com} shows that the log-belief is faster than linear consensus, which  is
		consistent with the theoretical  discussions  in Remark 2.

\begin{figure}
	\centering
	\includegraphics[width=0.8\columnwidth]{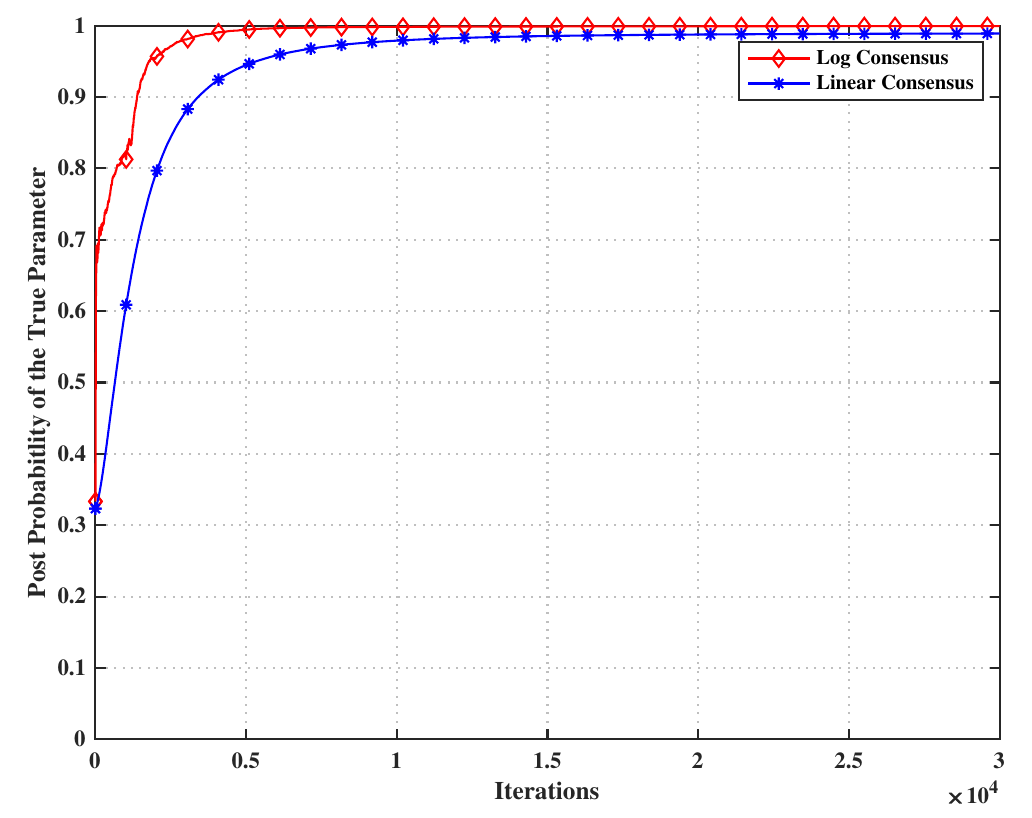}
	\caption{Comparison between   log-belief and  linear consensus }
	\label{log_com}
\end{figure}

	{\bf Different network topology comparison.}
	We demonstrate our algorithm in different distributed communication network and compare the performance under \rm{Erd\"os and R\'enyi} random graph \cite{erdos1960evolution} with different probability in network size of $N=30$. Since high probability of ER graph indicates more connectivity of network, we can obtain that the convergence under higher $p$ of ER graph is faster than that under lower $p$ of ER graph. The result is shown in Figure \ref{fig:topologycom}.
	
	\begin{figure}
		\centering
		\includegraphics[width=0.8\linewidth]{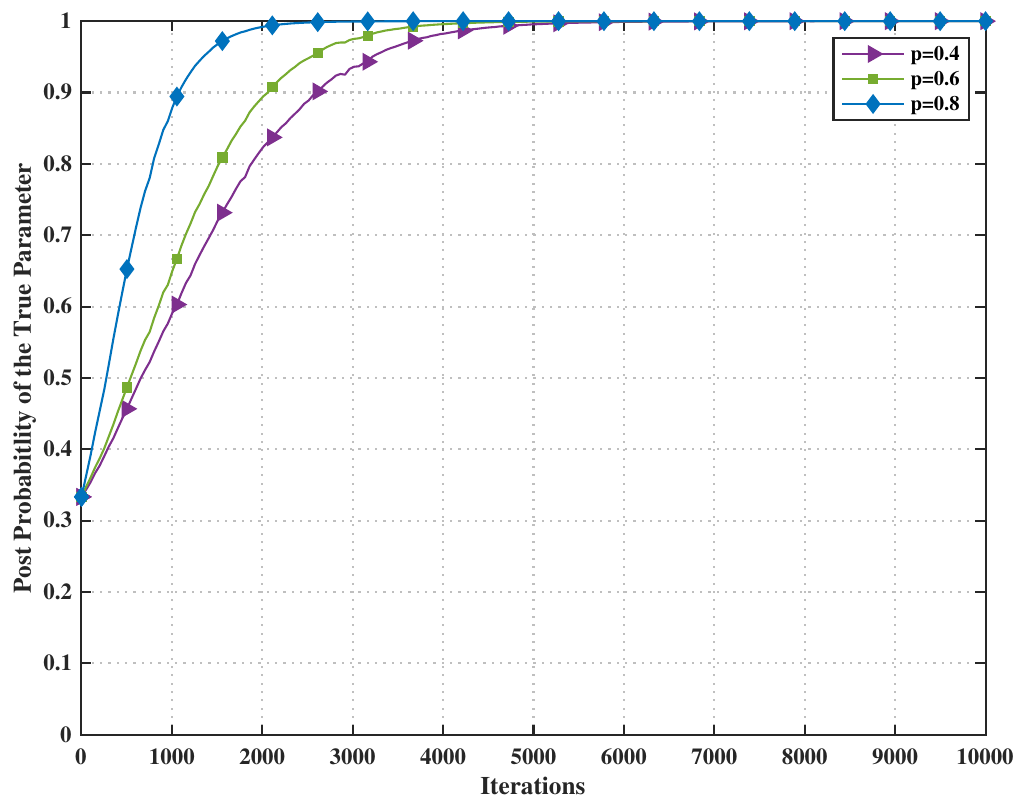}
		\caption{Comparison of Erd\"os-R\'enyi graph with different probability in network size of $N=30$.}
		\label{fig:topologycom}
\end{figure}

		\subsection{Source Seeking}\label{source_pro}
		
		In addition, we conduct  experiments on the  source seeking problems, which are of potential applications in gas leak detection and environmental protection \cite{chen2020combining, hutchinson2017review}.  Here we consider the steady-state plume model in two settings: ideal point source seeking without affect of ground, and point source above ground. 
		
		{\bf Ideal source seeking.}  Classical source seeking problems aim to find the source of atmospheric hazardous material and make the robot reach the source location eventually.  Consider distributed Unmanned Aerial Vehicle Networks (UAVNs) with five devices that are capable of sampling air quality, processing the data, and making decisions of the source localization based on the observations.

		
		Let $x=(x_{(1)},x_{(2)})$ denote  the localization of sampling air and $\theta=(\beta_{(1)},\beta_{(2)}) $  be   the pollution source localization.
		Under stable source strength and static conditions, the pollution source forms a stable field which is the Gauss model of continuous point source diffusion in unbounded space \cite{gotaas1972model} that can be formulated as
		\begin{align}
			c&(x=(x_{(1)},x_{(2)}) ;\theta=(\beta_{(1)},\beta_{(2)}))\notag\\
			& =c_0\exp{\left(-\frac{(x_{(1)}-\beta_{(1)})^2}{2\sigma_1^2}-\frac{(x_{(2)}-\beta_{(2)})^2}{2\sigma_2^2}\right)},\label{pollu}
		\end{align}
		where $c_0$ is the initial constant concentration of pollution source; $\sigma_1$ and $\sigma_2$ denote the lateral diffusion parameter and the longitudinal diffusion parameter, respectively.
		
		Five UAVs try to find the source by optimizing the aggregation function collaboratively
		\begin{equation}
			\min_{x}\frac{1}{5}\sum_{i=1}^{5}-c_i(x;\theta_*). \label{sp}
		\end{equation}
	where $c_i$ is defined in \eqref{pollu} with different $x_i$, which is the localization of different agents. The initial localizations  of five UAVs  are $(-2,0), (-0.5,3), (2,4), (4,-1),(1,-3)$, respectively. Three possible pollution source is $\theta_1=(0,0), \theta_2=(4,3), \theta_3=(2,-2)$, where $\theta_*=\theta_1$. Set $c_0=100, \sigma_1=\sigma_2=2$, and $\epsilon_i\sim N(0,1)$.
	
	The five UAVs use Algorithms \ref{alg:CDSA} to identify the true location of the target and adaptively move towards the center of the pollution source by detecting the concentration value at their current location. The  motion trajectories of the five  participants  are shown in Figure \ref{fig1}. It  can be  observed that all sensing and actuation devices first achieve consensus and then cooperatively locate the real pollution source. The experimental results align with theoretical analysis, indicating a faster convergence speed for consensus compared to the optimization convergence speed \cite[Remark 3]{mywork1}.
		\begin{figure}[h]
			\centerline{\includegraphics[width=0.9 \columnwidth]{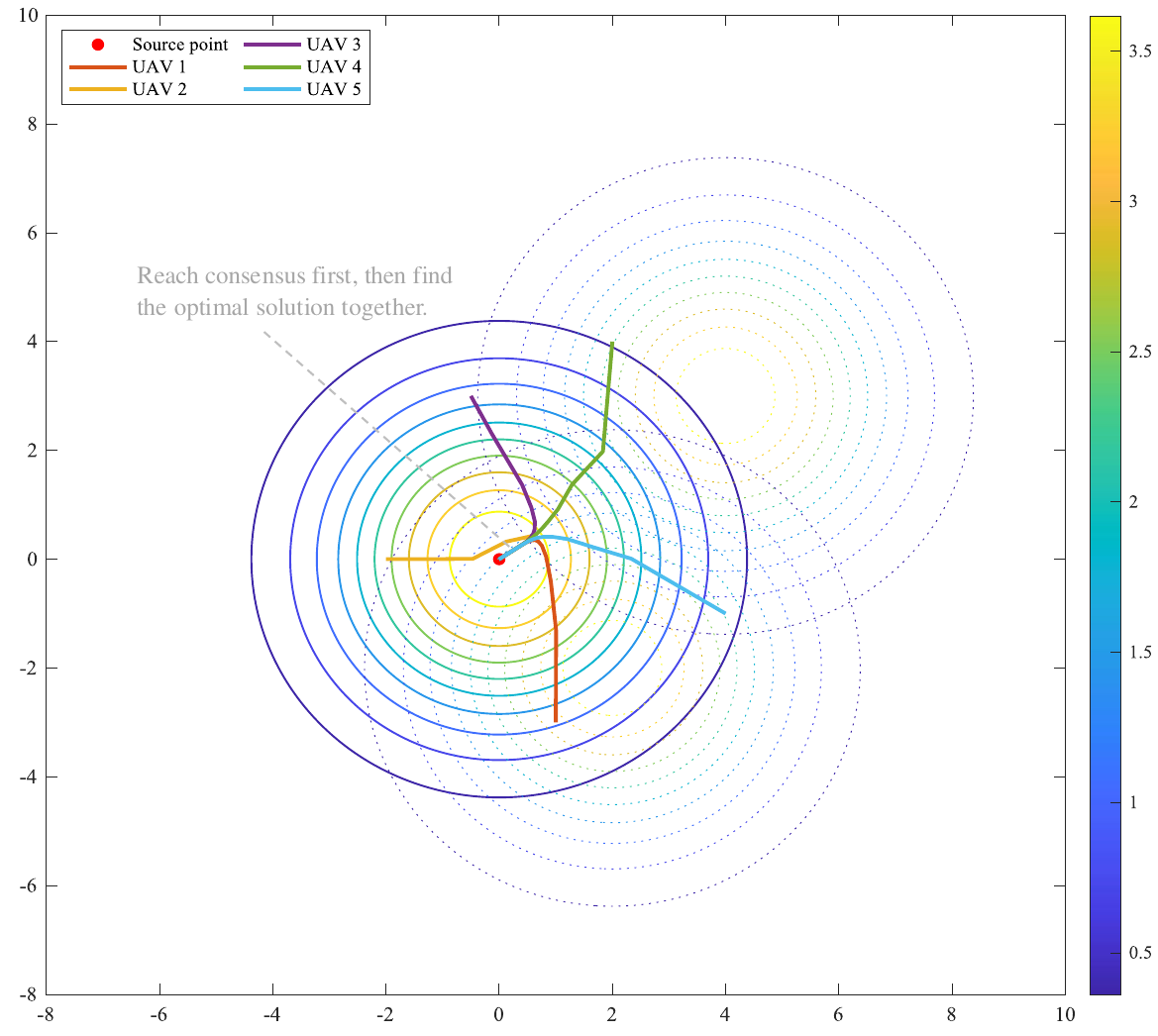}}
			\caption{Motion trajectories  of agents in ideal source seeking problem}
			\label{fig1}
		\end{figure}

 {\bf Point source above ground.} The influence of the ground surface on concentration distributions is incorporated by enforcing a zero-material-flux boundary condition at the terrain interface\cite{venkatram2003gaussian}. In this scenario, the concentration becomes  
	\begin{align}
		c(x=&(y,z); \theta=h)=\frac{Q}{2 \pi \sigma_{y} \sigma_{z} U} \exp \left[-\frac{y^2}{2 \sigma_{y}^2}\right]\notag\\
		&\times\left\{\exp \left[-\frac{\left(z-h\right)^2}{2 \sigma_{z}^2}\right]+\exp \left[-\frac{\left(z+h\right)^2}{2 \sigma_{z}^2}\right]\right\}, \label{h_s}
	\end{align}
 
where $Q$ is the source strength;  $U$ represents the time-averaged  wind speed at source height; $\sigma_y$ and $\sigma_z$ denote diffusion parameters in $y$ and $z$ directions and $h$ is  the height of the source above ground. Figure \ref{Schematic} is the schematic diagram 

	\begin{figure}[h]
	\centerline{\includegraphics[width=0.9 \columnwidth]{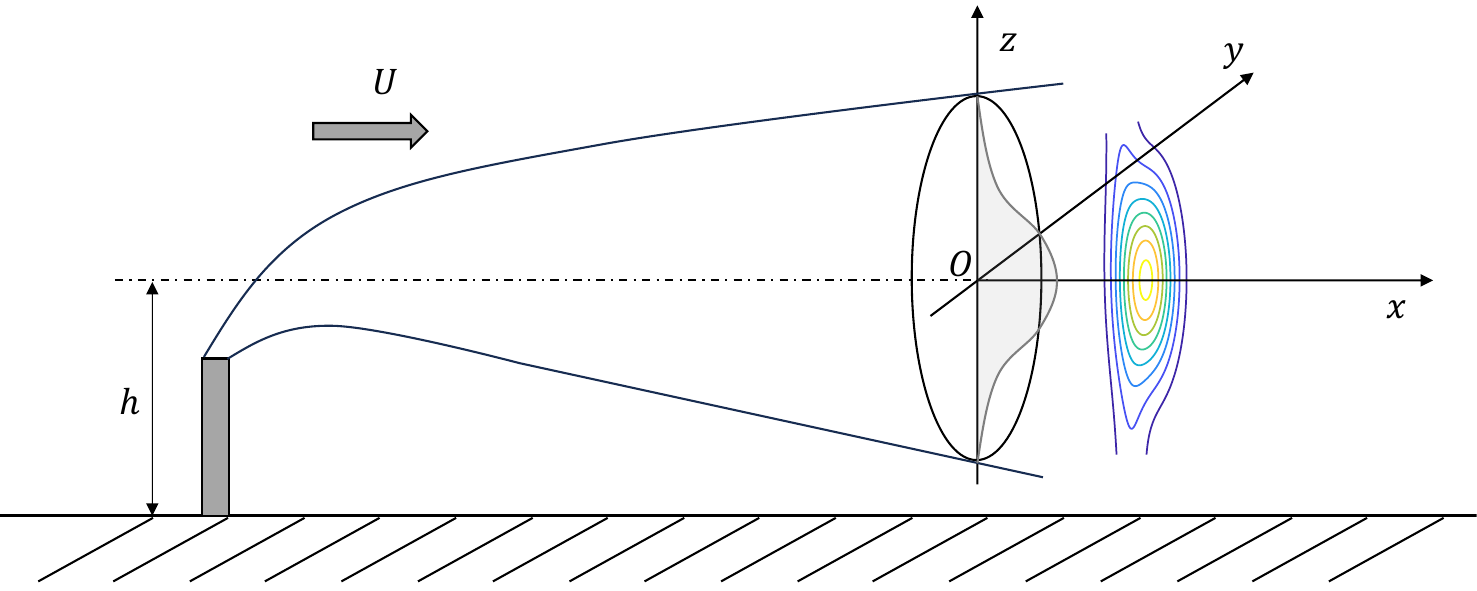}}
	\caption{Schematic diagram of air pollution model by point source above ground}
	\label{Schematic}
\end{figure}

The same as before, five UAVs try to find the unknown parameter source height $h$ by optimizing the aggregation function (\ref{sp}) with $c_i$ defined in (\ref{h_s}) under different $x_i$. Here $x_i=(y_i,z_i)$ is the projection of the UAVs position coordinates in the $yOz$ plane. The initial localizations  of five UAVs  are $(-5,2), (-2,10), (0,8), (3,5),(3.5,0)$ and the possible source height is $\theta_1=3$, $\theta_2=0.5$, $\theta_3=6$. Set $Q=500, \sigma_{y}=\sigma_{z}=2,$ and $U=3$. 

In this scenario, the goal of UAVs is arriving at consensus height which should be the source height. We conduct Algorithm \ref{alg:CDSA} to solve this problem and it really achieve the source height seeking goal as shown in Figure \ref{h_consentrate}. 

	\begin{figure}[h]
	\centerline{\includegraphics[width=0.9 \columnwidth]{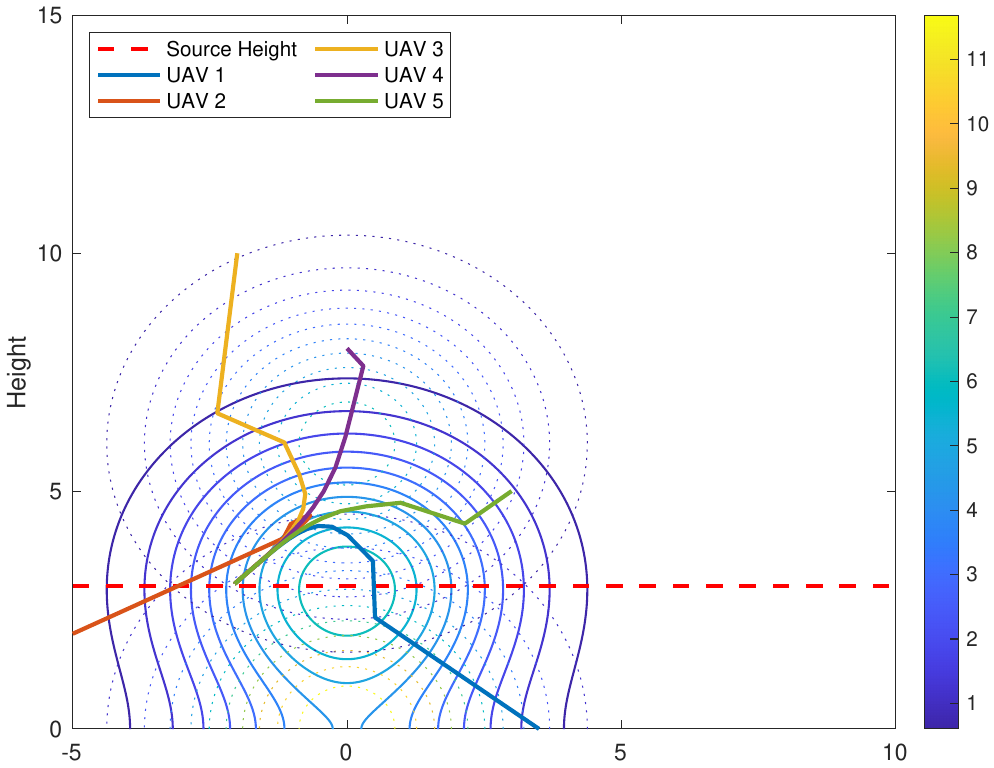}}
	\caption{Motion trajectories  of agents in point source seeking problem above ground}
	\label{h_consentrate}
\end{figure}


		\section{Conclusion}
		This work  has provided valuable insights for addressing  parametric  uncertainty in distributed optimization problems and simultaneously  finding  the optimal solution.
	To be specific, we have designed a novel distributed  fractional Bayesian  learning algorithm to resolve the  bidirectional coupled problem.  We then   prove that agents'  beliefs about the unknown  parameter converge to a common belief,  and that the decision variables also converge to the optimal solution almost surely.  It is worth noting from the  numerical  experiments that  by utilizing the   consensus protocol which  averages on a reweighting of the log-belief,  we have
	attained faster than normal distributed linear consensus protocol.   In future, we will further investigate the bidirectional  coupled  distributed optimization problems with continuous unknown model parameters.  In addition,  it is of interests to  consider the communication-efficient and other  distributed optimization methods  to such problems.
		
		\section*{References}
		\bibliographystyle{unsrt}
		\bibliography{references}
 \appendices

	\section{Proof of Lemma \ref{lemma_hatX}}
	\label{appe_hatX}
		\begin{proof}
		For any $t\geq0$, in order to bound $	\|\boldsymbol{x}^{(t)}-\boldsymbol{1}x^*(\tilde{\boldsymbol{q}})\|^2$, we firstly consider bounding $\|x_i^{(t)}-\alpha^{(t)}\nabla_xF_i(x_i^{(t)},\boldsymbol{q}_i^{(t)})-x^*(\tilde{\boldsymbol{q}})\|^2$ for all $i\in\mathcal{N}$.				\begin{align}
			&\|x_i^{(t)}-\alpha^{(t)}\nabla_xF_i(x_i^{(t)},\boldsymbol{q}_i^{(t)})-x^*(\tilde{\boldsymbol{q}})\|^2\notag\\
			&=\|x_i^{(t)}-x^*(\tilde{\boldsymbol{q}})\|^2+[\alpha^{(t)}]^2\|\nabla_xF_i(x_i^{(t)},\boldsymbol{q}_i^{(t)})\|^2\notag\\
			&~-2\alpha^{(t)}\nabla_xF_i(x_i^{(t)},\boldsymbol{q}_i^{(t)})^T(x_i^{(t)}-x^*(\tilde{\boldsymbol{q}}))\notag\\
			&\leq\|x_i^{(t)}-x^*(\tilde{\boldsymbol{q}})\|^2+[\alpha^{(t)}]^2\|\nabla_xF_i(x_i^{(t)},\boldsymbol{q}_i^{(t)})\|^2\notag\\
			&~-\!2\alpha^{(t)}\!\big(\nabla_xF_i(x_i^{(t)},\boldsymbol{q}_i^{(t)})\!-\!\nabla_xF_i(x^*(\tilde{\boldsymbol{q}}),\boldsymbol{q}_i^{(t)})\big)^T\!(x_i^{(t)}\!-x^*(\tilde{\boldsymbol{q}}))\notag\\
			&~+2\alpha^{(t)}\|\nabla_xF_i(x^*(\tilde{\boldsymbol{q}}),\boldsymbol{q}_i^{(t)})\|\|x_i^{(t)}-x^*(\tilde{\boldsymbol{q}})\|\notag\\
			&\leq(1-2\alpha^{(t)}\mu)\|x_i^{(t)}-x^*(\tilde{\boldsymbol{q}})\|^2+[\alpha^{(t)}]^2\|\nabla_xF_i(x_i^{(t)},\boldsymbol{q}_i^{(t)})\|^2\notag\\
			&\quad+2\alpha^{(t)}\|\nabla_xF_i(x^*(\tilde{\boldsymbol{q}}),\boldsymbol{q}_i^{(t)})\|\|x_i^{(t)}-x^*(\tilde{\boldsymbol{q}})\|\label{boundx_x}
		\end{align}
		where the last inequality follows by the strong convexity of $F_i(x,\boldsymbol{q})$ with $x$ in Lemma  \ref{lemma_allproperty}.
		
		Then similarly to the derivation of  \eqref{MJ} and \eqref{vec_F_bound}, we have
		\begin{align}
			\|\nabla_x&F_i(x_i^{(t)},\boldsymbol{q}_i^{(t)})\|^2\leq\notag\\
			&2M \|\nabla_x \boldsymbol{J}_i(x^*(\tilde{\boldsymbol{q}}), \boldsymbol{\theta})\|^2+2M^2L^2\|x_i^{(t)}-x^*(\tilde{\boldsymbol{q}})\|^2 \label{b1}
		\end{align}
		and
		\begin{align}
			\|\nabla_xF_i(x^*(\tilde{\boldsymbol{q}}),\boldsymbol{q}_i^{(t)})\|^2\leq M \|\nabla_x \boldsymbol{J}_i(x^*(\tilde{\boldsymbol{q}}), \boldsymbol{\theta})\|^2\label{b2}.
		\end{align}
		Substituting \eqref{b1} and \eqref{b2} into \eqref{boundx_x}, we can obtain
		\begin{align}
			\|x_i^{(t)}-&\alpha^{(t)}\nabla_xF_i(x_i^{(t)},\boldsymbol{q}_i^{(t)})-x^*(\tilde{\boldsymbol{q}})\|^2\notag\\
			&\leq(1-2\alpha^{(t)}\mu+2M^2L^2[\alpha^{(t)}]^2)\|x_i^{(t)}-x^*(\tilde{\boldsymbol{q}})\|^2\notag\\
			&~+2\alpha^{(t)}\|\nabla_x \boldsymbol{J}_i(x^*(\tilde{\boldsymbol{q}}), \boldsymbol{\theta})\|\sqrt{M}\|x_i^{(t)}-x^*(\tilde{\boldsymbol{q}})\|\notag\\
			&~+2M[\alpha^{(t)}]^2 \|\nabla_x \boldsymbol{J}_i(x^*(\tilde{\boldsymbol{q}}), \boldsymbol{\theta})\|^2.\notag
		\end{align}
		
		Since $\{\alpha^{(t)}\}_{t\geq 0}$ is a decreasing stepsize to zero, then there exists a constant $T>0$ such that for all $t\geq T$, $\alpha^{(t)}\leq\frac{\mu}{2M^2L^2}$.
		Hence, for any  $t\geq T$,
		\begin{align}
			\|x_i^{(t)}-&\alpha^{(t)}\nabla_xF_i(x_i^{(t)},\boldsymbol{q}_i^{(t)})-x^*(\tilde{\boldsymbol{q}})\|^2\notag\\
			&\leq(1-\alpha^{(t)}\mu)\|x_i^{(t)}-x^*(\tilde{\boldsymbol{q}})\|^2\notag\\
			&~+2\alpha^{(t)}\|\nabla_x \boldsymbol{J}_i(x^*(\tilde{\boldsymbol{q}}), \boldsymbol{\theta})\|\sqrt{M}\|x_i^{(t)}-x^*(\tilde{\boldsymbol{q}})\|\notag\\
			&~+\alpha^{(t)}\frac{\mu}{ML^2} \|\nabla_x \boldsymbol{J}_i(x^*(\tilde{\boldsymbol{q}}), \boldsymbol{\theta})\|^2\notag\\
			&\leq \|x_i^{(t)}-x^*(\tilde{\boldsymbol{q}})\|^2-\alpha^{(t)}\Big[\mu\|x_i^{(t)}-x^*(\tilde{\boldsymbol{q}})\|^2 \notag\\
			&~-2\sqrt{M}\|\nabla_x \boldsymbol{J}_i(x^*(\tilde{\boldsymbol{q}}), \boldsymbol{\theta})\|\|x_i^{(t)}-x^*(\tilde{\boldsymbol{q}})\|\notag\\
			&~-\frac{\mu}{ML^2} \|\nabla_x \boldsymbol{J}_i(x^*(\tilde{\boldsymbol{q}}), \boldsymbol{\theta})\|^2\Big]\label{xxx}.
		\end{align}
		Let us define 
		\begin{align}
			\mathcal{X}_i\triangleq\{p\geq0: \mu p^2&-2\sqrt{M}\|\nabla_x \boldsymbol{J}_i(x^*(\tilde{\boldsymbol{q}}), \boldsymbol{\theta})\|p\notag\\
			&-\frac{\mu}{ML^2} \|\nabla_x \boldsymbol{J}_i(x^*(\tilde{\boldsymbol{q}}), \boldsymbol{\theta})\|^2\leq0\}, \label{setX}
		\end{align}
		which is non-empty and compact. If $\|x_i^{(t)}-x^*(\tilde{\boldsymbol{q}})\|\notin \mathcal{X}$, we conclude from \eqref{xxx} that
		\begin{align}
			\|x_i^{(t)}-\alpha^{(t)}\nabla_xF_i(x_i^{(t)},\boldsymbol{q}_i^{(t)})-x^*(\tilde{\boldsymbol{q}})\|^2\leq \|x_i^{(t)}-x^*(\tilde{\boldsymbol{q}})\|^2.\label{bound1}
		\end{align}
		Otherwise,
		\begin{align}
			\|x_i^{(t)}&-\alpha^{(t)}\nabla_xF_i(x_i^{(t)},\boldsymbol{q}_i^{(t)})-x^*(\tilde{\boldsymbol{q}})\|^2\leq\notag\\
			&\max_{p\in\mathcal{X}_i}\Big\{p^2-\frac{\mu}{2M^2L^2}\big[\mu p^2-2\sqrt{M}\|\nabla_x \boldsymbol{J}_i(x^*(\tilde{\boldsymbol{q}}), \boldsymbol{\theta})\|p\notag\\
			&-\frac{\mu}{ML^2} \|\nabla_x \boldsymbol{J}_i(x^*(\tilde{\boldsymbol{q}}), \boldsymbol{\theta})\|^2\big]\Big\}\notag\\
			&=\max_{p\in\mathcal{X}_i}\Big\{(1-\tfrac{\mu^2}{2ML^2})p^2+\frac{\mu\|\nabla_x \boldsymbol{J}_i(x^*(\tilde{\boldsymbol{q}}), \boldsymbol{\theta})\|}{M^{1.5}L^2}p\notag\\
			&\qquad+\frac{\mu^2\|\nabla_x \boldsymbol{J}_i(x^*(\tilde{\boldsymbol{q}}), \boldsymbol{\theta})\|^2}{2M^3L^4}\Big\}. \label{bound2}
		\end{align}
		From the definition of $\mathcal{X}_i$ , the right zero point of the upward opening parabola in \eqref{setX} is
		\begin{align}
			&p_i^{(r)}=\frac{1}{2\mu}\Bigg(2\sqrt{M}\|\nabla_x \boldsymbol{J}_i(x^*(\tilde{\boldsymbol{q}}), \boldsymbol{\theta})\|\notag\\
			&+\sqrt{4M\|\nabla_x \boldsymbol{J}_i(x^*(\tilde{\boldsymbol{q}}), \boldsymbol{\theta})\|^2+\frac{4\mu^2}{ML^2} \|\nabla_x \boldsymbol{J}_i(x^*(\tilde{\boldsymbol{q}}), \boldsymbol{\theta})\|^2}\Bigg)\notag\\
			&=\left(\frac{\sqrt{M}}{\mu}+\frac{2}{L}\sqrt{\frac{M^2L^2+1}{M}}\right)\|\nabla_x \boldsymbol{J}_i(x^*(\tilde{\boldsymbol{q}}), \boldsymbol{\theta})\|,
		\end{align}
		which means $\mathcal{X}_i=[0, p_i^{(r)}]$. Since the values of quadratic function is bounded in a bounded closed set, we define
		\begin{align}
			&\max_{p\in\mathcal{X}_i}\Big\{(1- \tfrac{\mu^2}{ML^2})p^2+\frac{\mu\|\nabla_x \boldsymbol{J}_i(x^*(\tilde{\boldsymbol{q}}), \boldsymbol{\theta})\|}{M^{1.5}L^2}p\notag\\
			&\qquad+\frac{\mu^2\|\nabla_x \boldsymbol{J}_i(x^*(\tilde{\boldsymbol{q}}), \boldsymbol{\theta})\|^2}{2M^3L^4}\Big\}\triangleq R_i.\label{bound2_2}
		\end{align}
		Combining \eqref{bound1} and \eqref{bound2}, together with \eqref{bound2_2}, we have
		\begin{align}\label{bd-xi}
			\|x_i^{(t)}-&\alpha^{(t)}\nabla_xF_i(x_i^{(t)},\boldsymbol{q}_i^{(t)})-x^*(\tilde{\boldsymbol{q}})\|^2\notag\\
			&\leq \max\{\|x_i^{(t)}-x^*(\tilde{\boldsymbol{q}})\|^2, R_i\},  \quad \forall t\geq T.
		\end{align}
		Recalling from the definition of $\boldsymbol{x}^{(t)}$ and $\mathbb{F}(\boldsymbol{x}^{(t)},\boldsymbol{Q}^{(t)})$ in \eqref{mathbbx} and \eqref{mathbbF} respectively,  in light of relation \eqref{newx_iter} we have
		\begin{align*}
			\|\boldsymbol{x}^{(t+1)}&-\boldsymbol{1}x^*(\tilde{\boldsymbol{q}})\|^2\notag
			\\&=\|W\|^2\|\boldsymbol{x}^{(t)}-\alpha^{(t)}\mathbb{F}(\boldsymbol{x}^{(t)},\boldsymbol{Q}^{(t)})-\boldsymbol{1}x^*(\tilde{\boldsymbol{q}})\|^2\notag\\
			&\leq \|\boldsymbol{x}^{(t)}-\alpha^{(t)}\mathbb{F}(\boldsymbol{x}^{(t)},\boldsymbol{Q}^{(t)})-\boldsymbol{1}x^*(\tilde{\boldsymbol{q}})\|^2
			\\& \overset{\eqref{bd-xi}}{\leq}  \max\{\|\boldsymbol{x}^{(t)}-\boldsymbol{1}x^*(\tilde{\boldsymbol{q}})\|^2, \sum_{i=1}^{N}R_i\},  \quad \forall t\geq T,
		\end{align*}
		where the first inequality holds by the $2$-norm of $W$ is $1$ from Assumption \ref{assum_graph}.
		As a result,
		\begin{align} \label{bd-xT} \|\boldsymbol{x}^{(t)}-\!\boldsymbol{1}x^*(\tilde{\boldsymbol{q}})\|^2\!\leq\!\max\Big\{\|\boldsymbol{x}^{(T)}-\boldsymbol{1}x^*(\tilde{\boldsymbol{q}})\|^2, \sum_{i=1}^{N} R_i\Big\}.
		\end{align}

 Note that $\boldsymbol{x}^{(t+1)}=W\left(\boldsymbol{x}^{(t)}-\alpha^{(t)} \nabla_x \mathbb{F}\left(\boldsymbol{x}^{(t)}, \boldsymbol{Q}^{(t)}\right)\right)$ from \eqref{newx_iter} based on the vector form of $\boldsymbol{x}$ and $\mathbb{F}$ in \eqref{mathbbx} and \eqref{mathbbF}. Since  each belief value $\boldsymbol{q}_i^{(t)}(\boldsymbol{\theta})$  is bounded by $1$ for all $t\geq0$ and $i\in\mathcal{N}$, $\boldsymbol{Q}^{(t)}$  defined in  \eqref{big_Q} is bounded. This  together with  the
continuity of $\nabla_x\mathbb{F}$    under Assumption \ref{assum_func},
we conclude that for a fixed constant $T$, $\boldsymbol{x}^{(T)}$ is bounded. This together with \eqref{bd-xT}
proves the lemma.
	\end{proof}

\section{Proof of Lemma \ref{lemma_rec_e}}
\label{appe_rec_e}
\vspace{-5pt}
\begin{small}
	\begin{proof}
	According to the recursion  \eqref{rec-et}, we have
	\begin{align}
		e^{(t+1)}
		&\leq {\delta}^{t+1} e^{(0)}+ c\sum_{\tau=0}^t {\delta}^{t-\tau} [\alpha^{(\tau)}]^2\notag.\end{align}
	Since $\alpha^{(t)}$ is of order $\mathcal{O}(\frac{1}{t})$, without loss of generality we set $\alpha^{(t)}=\frac{\gamma}{t+T}$ with a constant $\gamma>0$ and $T>0$.
	Dividing both side of above inequality by  $[\alpha^{(t)}]^2$, we have
	\begin{align}
		{ e^{(t+1)} \over [\alpha^{(t)}]^2}
		&\leq { {\delta}^{t+1} \over [\alpha^{(t)}]^2} e^{(0)}+ c\sum_{\tau=0}^t {\delta}^{t-\tau} \left[{\alpha^{(\tau)} \over \alpha^{(t)}}\right]^2\notag\\
		&=\underbrace{\frac{e^{(0)}}{\gamma^2}\cdot\frac{\delta^{t+1}}{\frac{1}{(t+T)^2}}}_{Term~4}+ \underbrace{c\sum_{\tau=0}^t\delta^{t-\tau}\left(\frac{t+T}{\tau+T}\right)^2}_{Term~5}\label{term45}.
	\end{align}
\vspace{-5pt}
	As for $Term~4$, since $\frac{1}{\delta}>1$ by $\delta\in(0,1)$,   we can obtain that
	\begin{equation}
		\lim_{t \rightarrow \infty}\frac{\delta^{t+1}}{\frac{1}{(t+T)^2}}=\lim_{t \rightarrow \infty}\frac{(t+T)^2 }{(\frac{1}{\delta})^{t}}=0.\label{Term4}
	\end{equation}
	As for $Term~5$, we have
	\begin{align}
		Term~5&=c\sum_{\tau=0}^t\delta^{t-\tau}\left(1+\frac{t-\tau}{\tau+T}\right)^2\notag\\
		&=c\sum_{\tau=0}^t\delta^{t-\tau}\left(1+\frac{2(t-\tau)}{\tau+T}+\frac{(t-\tau)^2}{(\tau+T)^2}\right)\notag\\
		&\leq c\sum_{\tau=0}^t\delta^{\tau} +\frac{2c}{T}\sum_{\tau=1}^t \tau \delta^{\tau} + \frac{c}{T^2}\sum_{\tau=1}^t \tau^2 \delta^{\tau}.\label{term5}
	\end{align}
	Since $\delta\in(0,1)$, we derive
	\begin{align}
		&\lim_{t \rightarrow \infty} \sum_{\tau=0}^t\delta^{\tau}=\frac{1}{1-\delta},\label{lim1}\\
		&\lim_{t \rightarrow \infty} \sum_{\tau=1}^t \tau \delta^{\tau}=\delta	\sum_{\tau=1}^{\infty} \tau \delta^{\tau-1}= \delta\left( \sum_{\tau=0}^{\infty}\delta^{\tau}\right)^{\prime}=\frac{\delta}{(1-\delta)^2}. \label{lim2}
	\end{align}
	Moreover, due to
	\begin{align}
		\left(\sum_{\tau=1}^t\delta^{\tau}\right)^{\prime\prime}=\sum_{\tau=2}^{t}\tau(\tau-1)\delta^{\tau-2}=\sum_{\tau=2}^t \tau^2\delta^{\tau-2}-\sum_{\tau=2}^t \tau \delta^{\tau-2}\notag,
	\end{align}
	we can obtain that
	\begin{align}
		\sum_{\tau=1}^t\tau^2 \delta^{\tau}= \delta^2\left[	\left(\sum_{\tau=1}^t\delta^{\tau}\right)^{\prime\prime}+\frac{1}{\delta}\sum_{\tau=1}^t \tau \delta^{\tau-1}-1\right]+\delta,\notag
	\end{align}
	and therefore
	\begin{align}
		\lim_{t \rightarrow \infty} \sum_{\tau=1}^t \tau^2 \delta^{\tau}&= \delta^2\left[\frac{2}{(1-\delta)^3}+\frac{1}{\delta(1-\delta)^2}-1\right]+\delta\notag\\
		&=\frac{\delta(1+\delta)}{(1-\delta)^2}+\delta(1-\delta)\label{lim3}.
	\end{align}
	Substituting \eqref{lim1}, \eqref{lim2}, and \eqref{lim3} into \eqref{term5}, we can get the upper bound of $Term~5$. Together with \eqref{Term4} of $Term~4$ and recalling \eqref{term45},  we acheive
	\begin{align}
		0\leq\lim_{t \rightarrow \infty}\frac{e^{(t+1)}}{ [\alpha^{(t)}]^2}\leq c\left[1+\frac{\delta^2(\delta+3)}{T^2(1-\delta)^2}+\frac{2\delta}{T^2(1-\delta)}\right].\notag
	\end{align}
	Thus, $e^{(t+1)}=\mathcal{O}([\alpha^{(t)}]^2)=\mathcal{O}(\frac{1}{t^2})$, which yields the conclusion.
\end{proof}
\end{small}
\vspace{-20pt}
\begin{IEEEbiography}[{\includegraphics[width=1in,height=1.2in,clip,keepaspectratio]{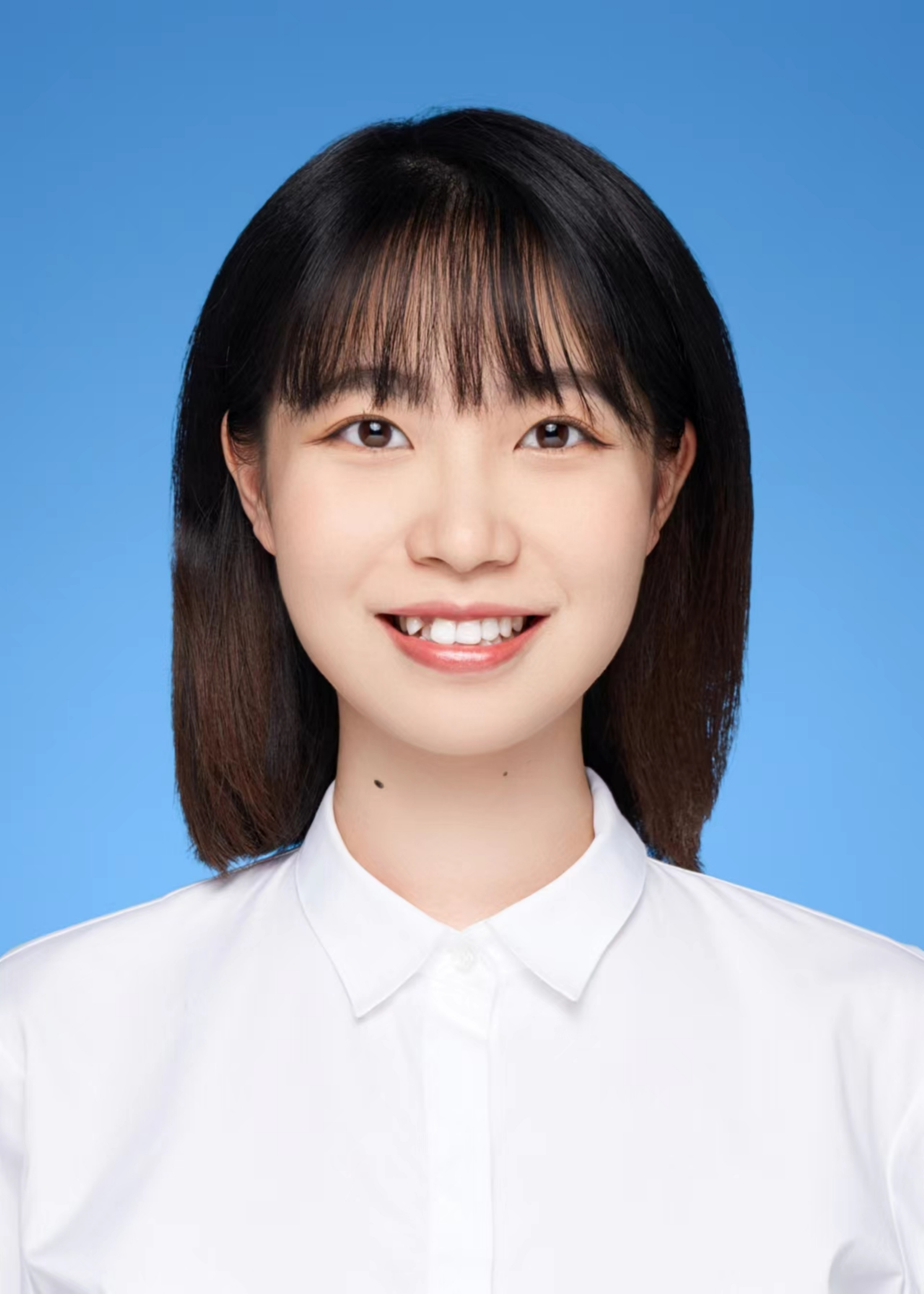}}]{Yaqun Yang} \begin{footnotesize}
 received the B.S. degree in China university of mining and technology, Beijing, China, in 2021. She is currently pursuing the Ph.D. degree with Tongji University, Shanghai, China.
	Her current research interests include distributed optimization with uncertainty and stochastic approximation.
\end{footnotesize}
\end{IEEEbiography}

	\vspace{-36pt}
	\begin{IEEEbiography} [{\includegraphics[width=1in,height=1.2in,clip,keepaspectratio]{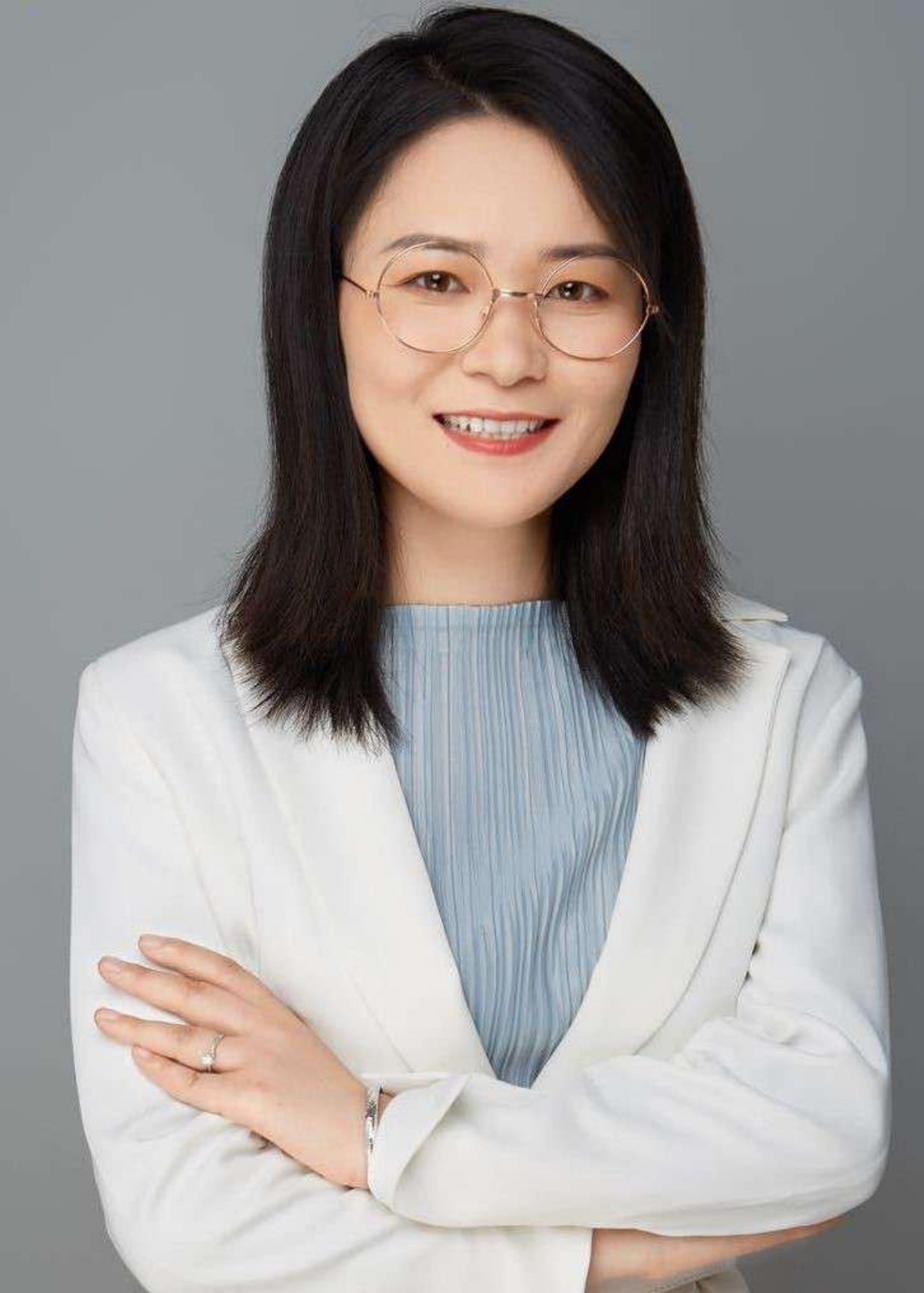}}]{Jinlong Lei} \begin{footnotesize}	(Member, IEEE) received the B.S. degree in automation from the University of Science and Technology of China, Hefei, China, in 2011, and the Ph.D. degree in operations research and cybernetics from the Institute of Systems Science, Academy of Mathematics and Systems Science, Chinese Academy of Sciences, Beijing, China, in 2016. She is currently a Professor with the Department of Control Science and Engineering, Tongji University, Shanghai, China. She was a postdoctoral fellow with the Department of Industrial and Manufacturing Engineering, Pennsylvania State University from 2016 to 2019. Her research interests include stochastic approximation, optimization and game-theoretic problems in networked regimes complicated by uncertainty.
			\end{footnotesize}
	\end{IEEEbiography}
\vspace{-31pt}
\begin{IEEEbiography}[{\includegraphics[width=1in,height=1.2in,clip,keepaspectratio]{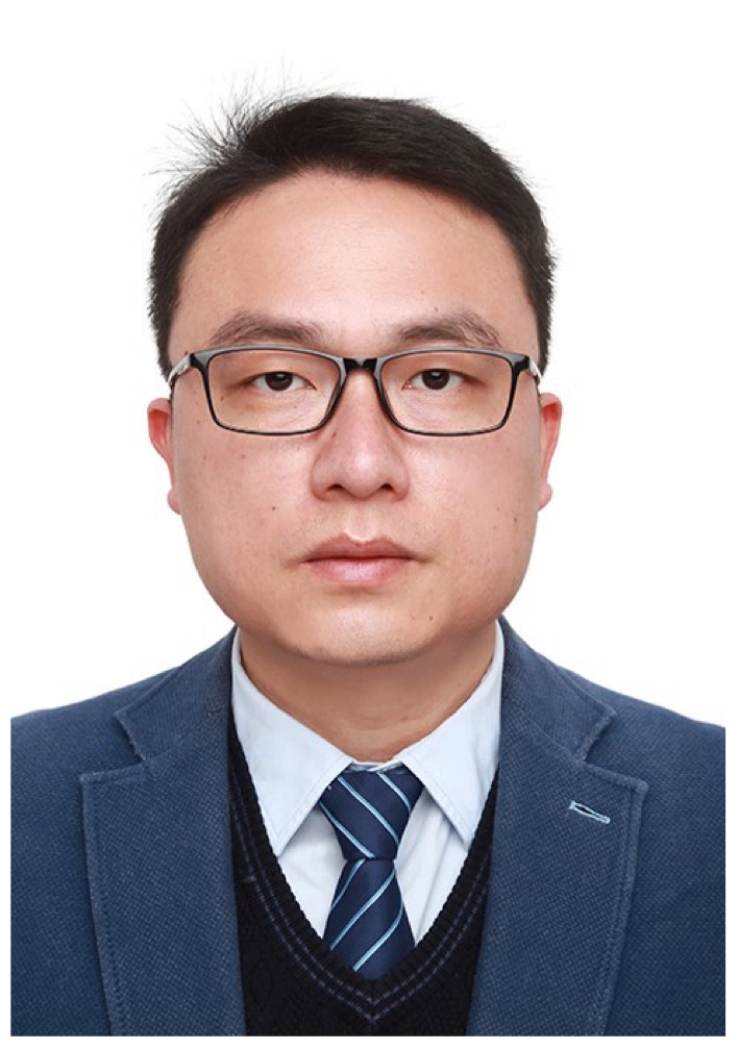}}]{Guanghui Wen}\begin{footnotesize}
 (Senior Member, IEEE) received the Ph.D. degree in mechanical systems and control from Peking University, Beijing, China, in 2012.
	Currently, he is a Young Endowed Chair Professor with the Department of Systems Science, Southeast University, Nanjing, China. His current research interests include autonomous intelligent systems, complex networked systems, distributed control and optimization, resilient control, and distributed reinforcement learning.
	Prof. Wen was the recipient of the National Science Fund for Distinguished Young Scholars, and Australian Research Council Discovery Early Career Researcher Award. He is a reviewer for American Mathematical Review and is an active reviewer for many journals. He currently serves as an Associate Editor of the IEEE Transactions on Industrial Informatics, the IEEE Transactions on Neural Networks and Learning Systems, the IEEE Transactions on Intelligent Vehicles, the IEEE Journal of Emerging and Selected Topics in Industrial Electronics, the IEEE Transactions on Systems, Man and Cybernetics: Systems, the IEEE Open Journal of the Industrial Electronics Society, and the Asian Journal of Control. Prof. Wen has been named a Highly Cited Researcher by Clarivate Analytics since 2018.
		\end{footnotesize}
\end{IEEEbiography}
\vspace{-35pt}

\begin{IEEEbiography}[{\includegraphics[width=1in,height=1.2in,clip,keepaspectratio]{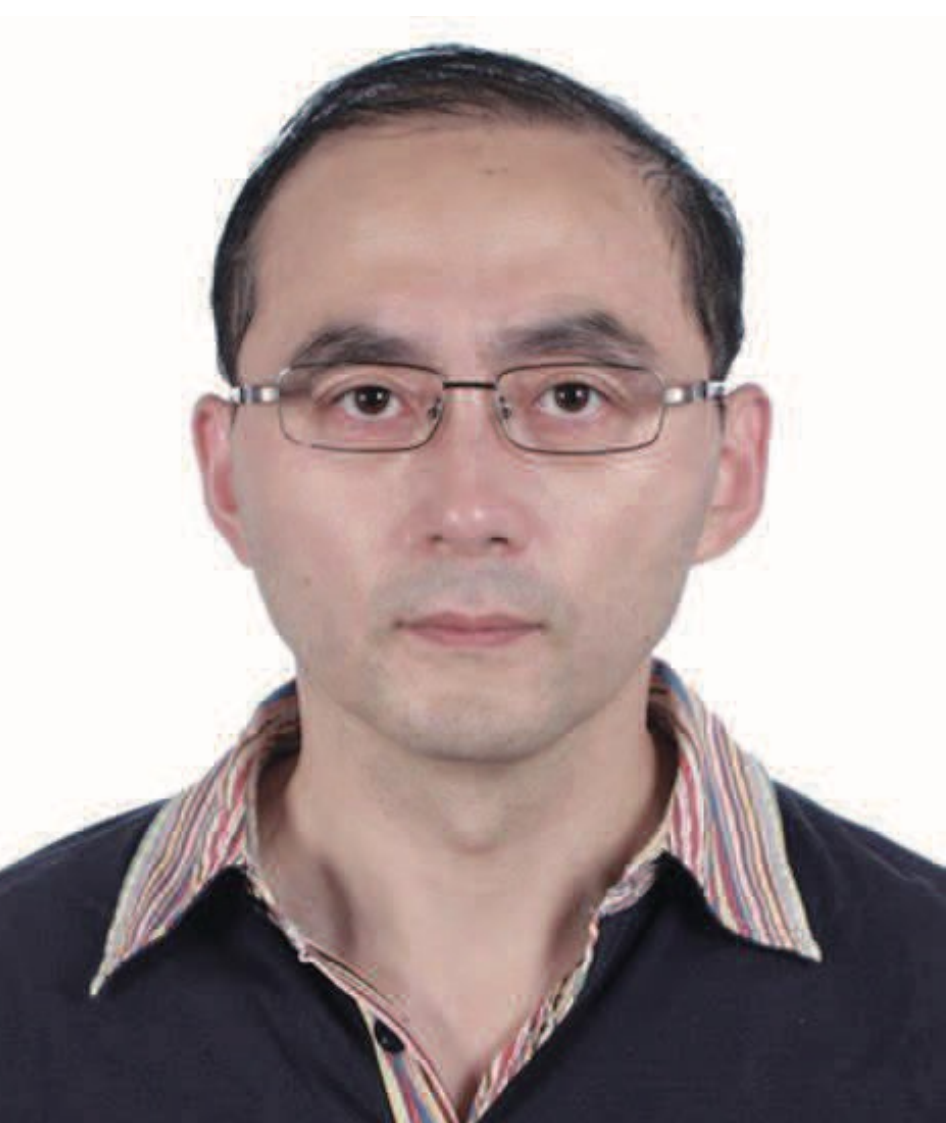}}]{Yiguang Hong} \begin{footnotesize}
(Fellow, IEEE) received the B.S. and M.S. degrees from Peking University, Beijing, China, and the Ph.D. degree from the Chinese Academy of Sciences (CAS), Beijing, China. He is currently a professor and deputy director of the Shanghai Research Institute for Intelligent Autonomous Systems, Tongji University, Shanghai, China, and an adjunct professor of Academy of Mathematics and Systems Science, CAS. He served as the director of the Key Lab of Systems and Control, CAS, and the director of the Information Technology Division, National Center for Mathematics and Interdisciplinary Sciences, CAS. His current research interests include nonlinear control, multi-agent systems, distributed optimization and game, machine learning, and social networks.

Dr. Hong is a Fellow of IEEE, a Fellow of the Chinese Association for Artificial Intelligence (CAAI), and a Fellow of the Chinese Association of Automation (CAA). Moreover, he was a recipient of the Guan Zhaozhi Award at the Chinese Control Conference, the Young Author Prize of the IFAC World Congress, the Young Scientist Award of CAS, the Youth Award for Science and Technology of China, and the National Natural Science Prize of China. Additionally, he was a Board of Governor of IEEE Control Systems Society. He serves or served as the Associate Editor for many journals, including the IEEE Transactions on Automatic Control, IEEE Transactions on Control of Network Systems, and IEEE Control System Magazine.
		\end{footnotesize}
\end{IEEEbiography}

	\end{document}